\documentclass{amsart}
\usepackage{array}
\usepackage{latexsym}
\usepackage[all]{xy}
\usepackage{mathrsfs}
\usepackage{graphics}
\usepackage{tikz,tikz-cd}
\usepackage{float}
\usepackage{enumitem}
\usepackage[colorlinks]{hyperref}

\newtheorem{theorem}{Theorem}[section]
\newtheorem{lemma}[theorem]{Lemma}

\theoremstyle{definition}

\newtheorem{example}[theorem]{Example}
\newtheorem{remark}[theorem]{Remark}

\theoremstyle{remark}

\newcommand{\tabincell}[2]{\begin{tabular}{@{}#1@{}}#2\end{tabular}} 

\begin{document}
 
	\title{The minimal and next minimal volumes of normal KSBA stable surfaces with $p_g\ge 2$         
	}
	\author{Jingshan \uppercase{Chen}}             
	\address{ school of mathematical science, USTC\\
		E-mail:\,$chjingsh@ustc.edu.cn$ }
	
	
	\maketitle

	\begin{abstract}
		In this paper we investigate the minimal and the next minimal volumes of normal KSBA stable surfaces with $p_g\ge 2$.
		We show that in case of  $|K_X|$ not composed with a pencil, the minimal and next minimal volumes are $2p_g-4$ and $2p_g-4+\frac{1}{3}$. In case of $|K_X|$  composed with a pencil, the minimal and next minimal volumes are $\frac{p_g-1}{p_g+1}(p_g-1)$ and  $\mathrm{min}\{\frac{2p_g-2}{2p_g+1}(p_g-1), \frac{(3p_g-2)p_g-4}{3(p_g+2)}\}$. We also characterize the surfaces achieving the minimal volumes.
	\end{abstract}

	\section{Introduction}
	Stable curves and stable pointed curves are by now fundamental objects in the study of compactification of moduli spaces of smooth curves of genus $g$. KSBA stable surfaces and stable log surfaces are the two-dimensional analogues (see  \cite{alexeev96a} \cite{alexeev06} \cite{kollar12} \cite{SMMP} \cite{kollarModuli} \cite{KSB88} ). It is known that the moduli space $\mathcal{M}_{a,b}$ of canonical models of smooth surfaces of general type with $a=K_X^2$ and $b=\chi(\mathcal{O}_X)$ has a KSBA compactification $\overline{\mathcal{M}}_{a,b}$, which parametrizing KSBA stable surfaces. 
	
	However, in general, a KSBA stable surface admits slc (semi-log-canonical) singularities, and its canonical divisor $K_X$ is only $\mathbb{Q}$-Cartier. Hence, the geometry of KSBA stable surfaces is quite different with that of smooth surfaces of general type. For example, a KSBA stable surface may be non-normal, and its volume $K_X^2$ is only a rational number. 
	It is well known that $2\chi(\mathcal{O}_X)-6\le K_X^2\le 9\chi(\mathcal{O}_X)$ for a smooth minimal surface of general type. A natural  question comes to mind:  what restrictions does 
	the volume $K_X^2$ of a KSBA stable surface have?
	This is a difficult question. First, it is easy to see $K_X^2>0$ since $K_X$ is required to be ample. 
	By a deep result of Alexeev \cite{alexeev94}, the volume set $$\mathcal{V}\colon=\{K_X^2\mid X \text{ is a KSBA stable surface}\}$$ is a DCC (descending chain condition) set. In particular, $\mathcal{V}$ has a minimal value $v_{min}$. The smallest known volume is $1/48983$, by an example of Alexeeve and Liu \cite[Theorem 1.4]{AL19}.

	
	In this paper, we consider the following subsets of $\mathcal{V}$:
	$$\mathcal{W}_{n}\colon=\{K_X^2\mid X \text{ is a normal KSBA stable surface with } p_g=n+1\ge 2\},$$  
	
	\begin{equation*}
		\mathcal{W}_{n}^{\circ}\colon= \left\{K_X^2\, \bigg| \,
		\begin{aligned}
			X \text{ is a }&\text{normal KSBA stable surface with } p_g=n+1\ge 2,\\
			&\text{\,and\,} |K_X| \text{\, not composed with a pencil}
		\end{aligned}\right\}.
	\end{equation*}
	
	Since $\mathcal{W}_{n}$ (resp. $\mathcal{W}_{n}^{\circ}$) is also a DCC set, it has a minimal value $w_{n,1}$ (resp. $w_{n,1}^{\circ}$)  and a next minimal value $w_{n,2}$ (resp. $w_{n,2}^{\circ}$). The first gap of $\mathcal{W}_{n}$ (resp. $\mathcal{W}_{n}^{\circ}$) is defined as $g_{n,1}\colon w_{n,2}-w_{n,1}$  (resp. $g_{n,1}^{\circ}\colon w_{n,2}^{\circ}-w_{n,1}^{\circ}$). There is natural question whether there is a uniform first gap $g_1$ (resp. $g_{n,1}^{\circ}$) for $\mathcal{W}_{n}$  (resp. $\mathcal{W}_{n}^{\circ}$) , i.e. $g_{n,1}=g_1$, $\forall n\ge 1$ (resp. $g_{n,1}^{\circ}=g_1^{\circ}$, $\forall n\ge 1$).
	
	The main theorem of this paper is as follows.
	
	\begin{theorem}[\textbf{Main theorem}]
		Let $X$ be a normal KSBA stable surface. 
		
		If $|K_X|$ is not composed with a pencil, then  $K_X^2\ge  2p_g-4$. If $K_X^2> 2p_g-4$, we have $K_X^2\ge  2p_g-4+\frac{1}{3}$. Moreover, the inequalities are sharp.
		
		If  $|K_X|$ is composed with a pencil, then either $K_X^2\ge 2p_g-2$, or $X$ is one of the following:
		\begin{itemize}
			
			\item[(A)]  $X$ is birational to a surface $\widetilde{X}$ which admits a genus 2 fibration over $\mathbb{P}^1$ with a section. Moreover, $p_g=2$ and $K_X^2\ge 1$. 
			If '$>$' holds, $K_X^2\ge 1+\frac{1}{3}$. The inequalities are sharp; 
			\item[(B)] $X$ is birational to a Jacobian surface over an elliptic curve. Moreover,  $K_X^2\ge p_g$. 
			If '$>$' holds, $K_X^2\ge p_g+\frac{1}{3}$. The inequalities are sharp;
			\item[(C)] $X$ is birational to an elliptic surface over $\mathbb{P}^1$. Moreover, $K_X^2\ge \frac{p_g-1}{p_g+1}(p_g-1)$. 
			If '$>$' holds, 
			$K_X^2 \ge \mathrm{min}\{\frac{2p_g-2}{2p_g+1}(p_g-1), \frac{(3p_g-2)p_g-4}{3(p_g+2)}\}$. The inequalities are sharp.
			
		\end{itemize}
	\end{theorem}
	
	Therefore, for the volume set $\mathcal{W}_{n}^{\circ}$ of normal KSBA stable surfaces with $|K_X|$ not composed with pencil, we have $w_{n,1}^{\circ}=2(n-1)$, $w_{n,2}^{\circ}= 2(n-1)+\frac{1}{3}$, and the first gap $g_{n,1}^{\circ}=\frac{1}{3}$ is uniform. For the volume set  $\mathcal{W}_{n}$ of normal KSBA stable surfaces, we have
	$w_{n,1}=\frac{n^2}{n+2}$, $w_{n,2}= \mathrm{min}\{\frac{2n^2}{2n+3}, \frac{3n^2+4n-3}{3(n+3)}\}$, and the first gap $g_{n,1}=\mathrm{min}\{\frac{n^{2}}{\left(n + 2\right) \left(2 n + 3\right)}, \frac{\left(n - 1\right) \left(n + 6\right)}{3 \left(n + 2\right) \left(n + 3\right)}\}$ is not uniform.
	
	The notion of KSBA stable surfaces naturally extend to the log case. 
	In the paper \cite{Liu22},  Wenfei Liu studied the minimal volume of KSBA stable log surfaces $(X,\Delta)$ of general type with $p_g(X,\Delta)\ge 1$. 
	
	The paper is organized as follows. In \textsection 2, we review some definitions and facts about KSBA stable surfaces. In \textsection 3, we consider the case where $|K_X|$ is not composed with a pencil. In \textsection 4, we classify normal KSBA stable surfaces with $|K_X|$  not composed with a pencil and $K_X^2=2p_g-4$. In \textsection 5, we consider normal KSBA stable surfaces with $|K_X|$ not composed with a pencil and $2p_g-4<K_X^2<2p_g-3$. In \textsection 6, we consider normal KSBA stable surfaces with $|K_X|$ composed with a pencil.

	\section{Preliminaries}\label{prel}
	
	\subsection{lc singularities}\label{clssoflc}
	A normal algebraic surface $X$ is called an KSBA stable surface if it admits only lc (log canonical) singularities and $K_X$ is a $\mathbb{Q}$-Cartier ample divisor. 
	
	We include here the classification of lc surface singularities from \cite[3.39]{SMMP}. 
	Let $s\in S$ be a lc germ with $\Gamma$ its dual graph. If $s$ is lt (log terminal), then $\Gamma$ is one of the following:
	
	\begin{itemize}[itemsep= 15 pt,topsep = 20 pt]
		\item (Cyclic quotient) $\Gamma$ is a string of rational curves as below:
		
		\begin{tikzpicture} 
			\begin{scope}[scale=1] 
				\clip(-1,-0.5) rectangle (8,0.5);
				
				\draw (0,0)--(1,0);
				\draw[dashed] (1,0)--(3,0);
				\draw (3,0)--(4,0);
				\node[circle, fill=white] at (0,0)  {$c_1$};
				\node[circle, fill=white] at (1,0)  {$c_2$};
				\node[circle, fill=white] at (3,0)  {$c_{n-1}$};
				\node[circle, fill=white] at (4,0)  {$c_n$};
				\draw node at (6, 0) {{}};
			\end{scope}
		\end{tikzpicture} 
		\item (Dihedral quotient) $\Gamma$ has a fork whose two branches are (-2)-curves; 
		
		\begin{tikzpicture} 
			\begin{scope}[ scale=1 ] 
				\clip(-1,-1.5) rectangle (8,1.5);
				\draw  (0,0.5)--(1,0);
				\draw  (0,-0.5)--(1,0);
				\draw (1,0)--(2,0);
				\draw[dashed] (2,0)--(4,0);
				\draw (4,0)--(5,0);
				\node[circle, fill=white] at (0,0.5)  {$2$};
				\node[circle, fill=white] at (0,-0.5)  {$2$};
				\node[circle, fill=white] at (1,0)  {$c_n$};
				\node[circle, fill=white] at (2,0)  {$c_{n-1}$};
				\node[circle, fill=white] at (4,0)  {$c_2$};
				\node[circle, fill=white] at (5,0)  {$c_1$};
			\end{scope}
		\end{tikzpicture} 
		
		\item (Other quotients) $\Gamma$ has a fork with each branch $\Gamma_i$ a string of rational curves.  
		
		\begin{tikzpicture}
			\begin{scope}[scale=1 ] 
				\clip(-1,-1) rectangle (4,2);
				\draw  (1,0)-- (1,1);
				\draw  (0,0)-- (2,0);
				\node[circle, fill=white] at (1,1)  {$\Gamma_1$};
				\node[circle, fill=white] at (0,0)  {$\Gamma_2$};
				\node[circle, fill=white] at (2,0)  {$\Gamma_3$};
				\node[circle, fill=white] at (1,0)  {$c_0$};
				
			\end{scope}
		\end{tikzpicture}
		
		There are three possibilities for $(\mathrm{det}(\Gamma_1),\mathrm{det}(\Gamma_2),\mathrm{det}(\Gamma_3))$:
		\begin{itemize}[leftmargin=90 pt]
			\item[(Tetrahedral)] (2,3,3);
			\item[(Octahedral)] (2,3,4);
			\item[(Icosahedral)] (2,3,5).
		\end{itemize}
	\end{itemize}

	If $s$ is not lt, then $\Gamma$ is one of the following:
	\begin{itemize}[itemsep= 15 pt,topsep = 20 pt]
		\item (Simple elliptic) $\Gamma=\{E\}$ has a single vertex which is a smooth elliptic curve with self-intersection $\le 1$;
		\item (Cusp) $\Gamma$ is a cycle of rational curve, $c_i\ge 2$ and there exists $c_i\ge 3$;
		
		
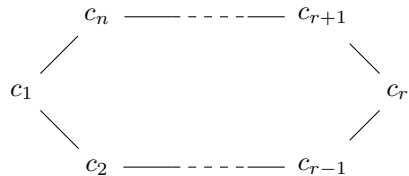
\begin{tikzpicture} 
			\begin{scope}[ scale=1 ] 
				\clip(-1,-1.5) rectangle (9,1.5);
				\draw  (0,0)--(1,1);
				\draw  (0,0)--(1,-1);
				\draw  (1,1)--(2,1);
				\draw  (3,1)--(4,1);
				\draw  (1,-1)--(2,-1);
				\draw  (3,-1)--(4,-1);
				
				\draw  (5,0)--(4,1);
				\draw  (5,0)--(4,-1);
				
				\draw[dashed] (2,1)--(3, 1);
				\draw[dashed] (2,-1)--(3,-1);
				\node[circle, fill=white] at (0,0)  {$c_1$};
				\node[circle, fill=white] at (1,-1)  {$c_2$};
				\node[circle, fill=white] at (4,-1)  {$c_{r-1}$};
				\node[circle, fill=white] at (4,1)  {$c_{r+1}$};
				\node[circle, fill=white] at (5,0)  {$c_r$};
				\node[circle, fill=white] at (1,1)  {$c_n$};

			\end{scope}
		\end{tikzpicture} 
		
		\item ($\mathbb{Z}/2$-quotient of a cusp or simple elliptic) $\Gamma$ has two forks with each branch a single (-2)-curve and $c_i\ge 2$;
		
		\begin{tikzpicture} 
			\begin{scope}[ scale=1 ] 
				\clip(-1,-1.5) rectangle (9,1.5);
				\draw  (0,0.5)--(1,0);
				\draw  (0,-0.5)--(1,0);
				\draw  (6,0.5)--(5,0);
				\draw  (6,-0.5)--(5,0);
				
				\draw (1,0)--(2,0);
				\draw[dashed] (2,0)--(4,0);
				\draw (4,0)--(5,0);
				\node[circle, fill=white] at (0,0.5)  {$2$};
				\node[circle, fill=white] at (0,-0.5)  {$2$};
				\node[circle, fill=white] at (6,0.5)  {$2$};
				\node[circle, fill=white] at (6,-0.5)  {$2$};
				\node[circle, fill=white] at (1,0)  {$c_1$};
				\node[circle, fill=white] at (2,0)  {$c_2$};
				\node[circle, fill=white] at (4,0)  {$c_{n-1}$};
				\node[circle, fill=white] at (5,0)  {$c_n$};

			\end{scope}
		\end{tikzpicture} 
		
		\item (Other quotients of a simple elliptic) $\Gamma$ has a fork with each branch $\Gamma_i$ a string of rational curves.  
		
		\begin{tikzpicture}
			\begin{scope}[scale=1 ] 
				\clip(-1,-1) rectangle (4,2);
				\draw  (1,0)-- (1,1);
				\draw  (0,0)-- (2,0);
				\node[circle, fill=white] at (1,1)  {$\Gamma_1$};
				\node[circle, fill=white] at (0,0)  {$\Gamma_2$};
				\node[circle, fill=white] at (2,0)  {$\Gamma_3$};
				\node[circle, fill=white] at (1,0)  {$c_0$};	
			\end{scope}
		\end{tikzpicture} 
		
		There are three possibilities for $(\mathrm{det}(\Gamma_1),\mathrm{det}(\Gamma_2),\mathrm{det}(\Gamma_3))$:
		\begin{itemize}[leftmargin=90pt]
			\item[($\mathbb{Z}/3$-quotient)] (3,3,3);
			\item[($\mathbb{Z}/4$-quotient)] (2,4,4);
			\item[($\mathbb{Z}/6$-quotient)] (2,3,6).
		\end{itemize}

	\end{itemize}
	
	The cases $\mathrm{det}(\Gamma_i)\in \{2,3,4,5,6\}$ gives the possibilities:
	\begin{align*}
		\mathrm{det}(\Gamma_i)=2 \quad  &\Leftrightarrow\quad \Gamma_i \quad\text{is }\  2;\\
		\mathrm{det}(\Gamma_i)=3\quad   &\Leftrightarrow\quad \Gamma_i \quad\text{is }\  3 \quad\text{or}\quad 2 - 2;\\
		\mathrm{det}(\Gamma_i)=4\quad   &\Leftrightarrow\quad \Gamma_i \quad\text{is }\  4 \quad\text{or}\quad 2 - 2 - 2;\\
		\mathrm{det}(\Gamma_i)=5\quad   &\Leftrightarrow\quad \Gamma_i \quad\text{is }\  5\quad\text{or}\quad 2 - 2 - 2 -2 \quad\text{or}\quad 2 - 3;\\
		\mathrm{det}(\Gamma_i)=6\quad   &\Leftrightarrow\quad \Gamma_i \quad\text{is }\  6 \quad\text{or}\quad 2 - 2 - 2 - 2 - 2.\\
	\end{align*}
	
	\begin{remark}
		Note that a lc singularity is a canonical singularity if and only if its exceptional curves  consists of (-2)-curves. 
		This can be shown by taking a minimal resolution $\pi\colon\widetilde{X}\to X$ with $\pi^*K_S=K_{\widetilde{X}}-\sum a_E E$. If each $E$ is a (-2)-curve, we have $0=-\sum a_E E\pi^*K_S=(-\sum a_E E)^2$. Hence $-\sum a_E E=0$, which means the singularity is a canonical singularity. 
	\end{remark}
	\begin{remark}\label{discrep}
		Let $\pi:\tilde{X}\to X$ be a resolution of a lc singularitie and $E'$ be an exceptional curve with $-E'^2=d\ge 3$ and $E'\cong \mathbb{P}^1$. Then we have $$0=E\pi^*K_S=E'(K_{\widetilde{X}}-\sum a_E E)=-2-E'^2-a_{E'}E'^2-\sum\limits_{E\not=E'} a_E EE'.$$
		Therefore, 
		$$-a_{E'}=\frac{-2-E'^2-\sum\limits_{E\not=E'} a_E EE'}{-E'^2}\ge \frac{-2-E'^2}{-E'^2}=\frac{d-2}{d}.$$
		'=' holds, if and only if $EE'=0$ for all $E\not=E'$, i.e. the exceptional locus of $\pi$ consists only of $E'$.
		
	\end{remark}
	
	\subsection{the numerical fundamental cycles of smooth point, $ADE$ singularities and elliptic Gorenstein singularities}\label{fundmtclcle}
	
	Let $x\in X$ be an $ADE$ singularity on a normal surface. Let $\pi:Y \to X$ be a minimal resolution of singularity. There exists a unique minimal effective divisor $Z:=\sum_{k=1}^{n} r_kC_j>0$ called the \emph{numerical fundamental cycle} or \emph{fundamental cycle} 
	such that $ZC_j\le 0$, $k=1,...,n$. 
	Moreover, $\mathrm{Supp}Z=\mathrm{Exc}(\pi)$, and $\mathfrak{M}_x \mathcal{O}_Y =\mathcal{O}_Y(-Z)$, where $\mathfrak{M}_x$ is the maximal ideal sheaf of $x$ in $X$. (See \cite{Reid97} 4.17, or \cite{BHPV}  {III. \textsection 3}).

	Let $x\in X$ be a simple elliptic singularity or a cusp on a normal surface, which is an elliptic Gorenstein singularity (we refer to \cite{Reid97} 4.21 for definition). 
	Let $\pi:Y \to X$ be a minimal resolution of singularity. Similarly to the $ADE$-singularity, the associated numerical fundamental cycle $Z$ is defined.
	On the other hand, we have $K_Y=\pi^*K_X-\sum a_EE=\pi^*K_X+\sum E$, where $\sum E$ denoted by $Z_K$ is called the canonical cycle. $Z_K$ is either a smooth elliptic curve when $x$ is an simple elliptic singularity, or a cycle of ration curve when $x$ is a cusp. 
	We have $Z=Z_K$(see \cite{Reid97} 4.21). 
	Set $d:=-Z^2$, defined to be the degree of $x$.
	If $d\ge 2$, $\mathfrak{M}_x\mathcal{O}_Y=\mathcal{O}_Y(-Z)$. If $d=1$, $\mathfrak{M}_x\mathcal{O}_Y=\mathfrak{M}_p\mathcal{O}_Y(-Z)$, where $p$ is a smooth point on  $Z$.
	(See \cite{Reid97} 4.23, 4.25.)
	
	Let $x\in X$ be a smooth point and $\pi:Y \to X$ be the blowup at $x$ with $E$ be the $(-1)$-curve. 
	The associated numerical fundamental cycle $Z$ is defined. It is evident that $Z=E$ and $\mathfrak{M}_x \mathcal{O}_Y =\mathcal{O}_Y(-Z)$.
	
	Let $x\in X$ be a smooth point, an $ADE$ singularity or an elliptic Gorenstein singularity and $Z$ be its numerical fundamental cycle. 
	Let $|L|$ is a movable linear system of Cartier divisors on a normal surface $X$ with $x$ contained in the base locus $\mathrm{Bs}|L|$.
	Write $\pi^*|L|=|\bar{L}|+G$, where $G$ is the base part.
	We have $G\ge Z$ by the previous discussion.
	Moreover $\bar{L}G\ge \bar{L}Z>0$, since $\mathrm{Supp}Z=\mathrm{Exc}(\pi)$ and $\bar{L}\cap \mathrm{Exc}(\pi)\not=\emptyset$.

	\subsection{lc singularities on a double covering surface of a smooth surface}
	Let $f\colon X\to S$ be a double covering map with $S$ smooth, $X$ normal and $K_X$ Cartier. Let $R$ be the ramification divisor, i.e. $K_X=f^*K_S+R$. The branched locus  $B=f_*R$. If $B$ is smooth, then $X$ is smooth. If $p$ is a singularity of $B$, then the preimage $q=f^{-1}(p)$ is an isolated singularity of $X$. 
	%
	If $q$ is a log canonical singularity on $X$, then $p$ is a simple singularity of $B$, a [3,3]-point (namely a triple point with an infinitely near triple point $p_1$ and all the points infinitely near to $p_1$ are at most double), or a quadruple point such that every point infinitely near to $p$ is at most double.
	We will call such curve singularities on $B$ are of lc type. Their corresponding singularities on $X$ are canonical singularities, elliptic Gorenstein singularities of degree 1 and elliptic Gorenstein singularities of degree 2. See \cite{FPR15a} 4.B or \cite{LR12} and \cite[6.5]{KSC04}.

	\section{Noether type inequality for stable surface with $|K_X|$ not composed with a pencil}\label{ncp}
	
	Let $X$ be a normal KSBA stable surface and $K_X$ be its canonical divisor. When $p_g\ge2$, the linear system $|K_X|$ defines a rational map $\phi_{|K_X|} \colon X \dashrightarrow \mathbb{P}^{p_g-1}$. We say $|K_X|$ is composed with a pencil (resp. not composed with a pencil) if the image $\phi_{|K_X|} (X)$ is a curve (resp. a surface). 

	Let $\pi\colon \widetilde{X} \to X$ be a resolution of singularities where $K_X$ is not Cartier. 
	$K_{\tilde{X}}$ is Cartier and there are at worst $ADE$ singularities and elliptic Gorenstein singularities on $\widetilde{X}$.  
	We have the discrepancy formula
	$$K_{\tilde{X}}=\pi^*K_X+\sum a_EE$$
	with $-1\le a_E<0$. 
	By definition, $p_g(X)=h^0(X,K_X)=h^0(\widetilde{X},\lfloor K_{\widetilde{X}}-\sum a_EE\rfloor)=h^0(\widetilde{X}, K_{\widetilde{X}}+\sum\limits_{a_E=-1} E)$.
	The curve $C:=\sum\limits_{a_E=-1} E$ is reduced and rational as the singularities resolved by $\pi$ are log terminal. 
	We have $h^0(\widetilde{X}, K_{\widetilde{X}}+C)=h^0(\widetilde{X}, K_{\widetilde{X}})$ by considering the associated cohomology sequence of $0\to \omega_X\to \omega(C)\to \omega_{C}\to 0$. Therefore $p_g(X)=p_g(\widetilde{X})$. Since $-\sum a_EE\ge 0$ and $K_{\tilde{X}}E\ge 0$, we have
	$$K_X^2=\pi^*K_X^2=K_{\tilde{X}}\cdot\pi^*K_X=K_{\tilde{X}}^2-K_{\tilde{X}}\sum a_EE\ge K_{\tilde{X}}^2.$$
	If $K_X$ is not Cartier, i.e. $\epsilon\not =\mathrm{id}$, then there exists an exceptional curve $E$ with $a_E<0$ and $K_{\tilde{X}}E>0$ by \ref{clssoflc}, and thus $K_X^2>K_{\tilde{X}}^2$.
	
	The linear system $|K_{\tilde{X}}|=|W|+V$, where $V$ the fixed part. Note that $W$ and $V$ are only  Weil-divisors. 
	Let $\epsilon\colon \bar{X}\to \widetilde{X}$ be a succession of resolutions and blowups of base points of the movable part of the pullbacks of $|W|$, such that $\epsilon^*K_{\widetilde{X}}=|L|+G$ where $G$ is the base part and the movable part $|L|$ is base-point-free. We see that $\epsilon^*K_{\widetilde{X}}=K_{\bar{X}}+\sum Z_i^e-\sum \mathcal{E}_k$, where $Z_i^e$'s are the pullbacks of the numerical fundamental cycles of  elliptic Gorenstein singularities and $\mathcal{E}_i$'s are the pullbacks of $(-1)$ curve of the blowups. 
	We have the following commutative diagram:
	\[
	\xymatrix{
		\bar{X}\ar[drr]_{\phi_{|L|}}\ar[r]^{\epsilon} &\tilde{X}\ar[r]^{\pi}\ar@{-->}[dr]^{\phi_{|K_{\tilde{X}}|}} &X\ar@{-->}[d]^{\phi_{|K_X|}}\\
		&&\mathbb{P}^{p_g-1}
	}\]
	Note that a general $L\in |L|$ is  the strict transformation of some $W\in |V|$ and $G= \overline{V}+R$, where $R$ is an effective $\epsilon$-exceptional divisor. 
	Moreover, if $\epsilon \not= \mathrm{id}$, by \textsection\ref{fundmtclcle} we have $G\ge \sum Z_i^e+ \sum Z_j^c+\sum \mathcal{E}_k$, where   $Z_j^c$'s are the pull-backs of numerical fundamental cycles of $ADE$ singularities resolved by $\epsilon$. 
	\begin{lemma}\label{firstineq}
		With notations as above, we have
		\begin{align*}
			K_X^2=\epsilon^*\pi^*K_X^2=&(L+G)\epsilon^*\pi^*K_X \\ 
			=&L(L+G-\sum a_E\epsilon^*E)+G\cdot\epsilon^*\pi^*K_X\\
			=& L^2+LG-\sum a_E\epsilon^*EL+G\cdot\epsilon^*\pi^*K_X\\
			\ge& L^2.
		\end{align*}

		Moreover, the following statements hold true:
		\begin{itemize}
			\item if $LG=0$, then $\epsilon=\mathrm{id}$;
			\item  if $LG=G\cdot\epsilon^*\pi^*K_X=0$, then $\epsilon=\mathrm{id}$ and $G=0$. 
			\item if $K^2=L^2$, then $\pi,\epsilon=\mathrm{id}$ and $G=0$, which means $K_X$ is Cartier and $|K_X|$ is base-point-free. 
		\end{itemize}
	\end{lemma}
	\begin{proof}
		Since $G\ge 0$ and $K_X$ is ample, we have $G\cdot\epsilon^*\pi^*K_X\ge 0$. Since $G\ge 0$,  $-\sum a_EE\ge 0$ and $L$ is nef, we have $LG\ge 0$ and $-\sum a_E\epsilon^*EL\ge 0$.
		Therefore, the inequality holds true.
		
%
		If $\epsilon\not=\mathrm{id}$, we have $LG\ge LR\ge L(\sum Z_i^e+ \sum Z_j^c+\sum \mathcal{E}_k)>0$.
		If further $G\cdot \pi^*K_X=0$, then $G$ is $\pi$-exceptional and thus $GK_{\tilde{X}}\ge 0$. We have 
		$$0\le K_{\tilde{X}}G=(L+G)G=G^2\le 0.$$
		As a result, we have $G=0$. 
		
		If $K^2=L^2$ holds, then $LG=G\cdot\epsilon^*\pi^*K_X=0$ and $-\sum a_E\epsilon^*EL=0$. 
		By the previous argument, $LG=G\cdot\epsilon^*\pi^*K_X=0$ implies $\epsilon=\mathrm{id}$ and $G=0$. Thus, $-\sum a_E\epsilon^*EL=0$ is equivalent to $-\sum a_EEK_{\tilde{X}}=0$, which implies $-\sum a_EE=0$ and thus $\pi=\mathrm{id}$. The last statement of the theorem follows.
		
	\end{proof}
	\begin{lemma}\label{castelnuovo}
		If $\phi_{|K_X|}$ is birational, then $K_X^2\ge 3p_g-7$.
		
	\end{lemma}
	\begin{proof}
		We have $K_{\bar{X}}=L+G-\sum Z_i^e+\sum \mathcal{E}_k$. Since $G\ge \sum Z_i^e+ \sum Z_j^c+\sum \mathcal{E}_k$, the divisor $H:=G-\sum Z_i^e+\sum \mathcal{E}_k\ge 0$. By Bertini's theorem we can take an irreducible smooth curve $C\in|L|$. By adjunction formula we have $K_C=(2L+H)|_C$. 
		The remaining proof is a mimick of \cite[Lemma 1.1]{Hor76b}. We include it here in the following for the sake of clarity and completeness. 
		
		Let $r=h^0(C,L|_C)$ and $\Gamma=p_1+...+p_{r-2}$ be a sufficient general effective divisor on $C$ such that $h^0(C,L|_C-\Gamma)=2$ and $h^0(C,K_C-\Gamma)=g(C)-(r-2)$. Since $|L|_C|$ defines a birational map of $C$, there exist  $\vartheta,\vartheta'\in |L|_C|$ such that $\mathrm{inf} (\vartheta,\vartheta')=\Gamma$. Moreover, we may assume $\vartheta+\vartheta'$ is disjoint from $H|_C$. Denote $\Omega(K_C-\vartheta)\colon=\mathrm{im}(H^0(K_C-\vartheta)\hookrightarrow H^0(K_C))$. we have the following linear map:
		\begin{align*}
			\gamma\colon \Omega(K_C-\vartheta) &\oplus \Omega(K_C-\vartheta'-H|_C) \rightarrow  \Omega(K_C-\Gamma)\\
			&(\phi,\phi') \mapsto \phi+\phi'.
		\end{align*}
		Let $(\phi,\phi')\in \mathrm{ker}\gamma$. Then $\phi=-\phi'$ vanishes on $\mathrm{sup}(\vartheta,\vartheta'+H|_C)=(\vartheta+\vartheta'+H|_C-\Gamma)\sim K_C-\Gamma$. Hence $\mathrm{ker}\gamma\cong H^0(C,K_C-(K_C-\Gamma))\cong H^0(C,\Gamma)\cong \mathbb{C}$. 
		On the other hand we have 
		\begin{align*}
			&\mathrm{dim}\Omega(K_C-\vartheta)=h^0(C,K_C-\vartheta)=r-\mathrm{deg} \vartheta+g(C)-1,\\
			&\mathrm{dim}\Omega(K_C-\vartheta'-H|_C)=h^0(C,L|_C)=r,\\
			&\mathrm{dim}\Omega(K_C-\Gamma)=g(C)-(r-2).
		\end{align*}
		
		Therefore we obtain the inequality $\mathrm{deg} \vartheta\ge 3r-4 $. As $r\ge h^0(\bar{X},L)-1=h^0(\widetilde{X},K_{\widetilde{X}})-1=p_g(X)-1$, 	we have $K_X^2\ge L^2$ by Lemma \ref{firstineq}. Since $L^2=\mathrm{deg} \vartheta$, it follows that $K_X^2\ge L^2\ge 3p_g-7$.
		
	\end{proof}
	
	\begin{theorem}
		Let $X$ be a normal KSBA stable surface with $|K_X|$ not composed with a pencil. Then $K_X^2\ge 2p_g-4$. Moreover if '=' holds, then $K_X$ is Cartier, $|K_X|$ is base-point-free, and $\phi_{|K_X|}$ defines a double covering map onto a surface of minimal degree. 
	\end{theorem}
	\begin{proof}
		If $\phi_{|K_X|}$ is birational, then by Lemma \ref{castelnuovo} we have
		\begin{align}\label{bir}
			K_X^2\ge L^2\ge 3p_g-7=2p_g-4+p_g-3\ge 2p_g-4. 
		\end{align}
		
		It remains to consider the case $\mathrm{deg}\phi_{|K_X|}\ge 2$. 
		Let $W$ be the image of $\phi_{|L|}\colon \bar{X}\to \mathbb{P}^{p_g-1}$. We have $\mathrm{deg}W\ge p_g-2$. 
		In fact, a nondegenerate $n$-dimensional projective irreducible variety $Z\subset \mathbb{P}^{m}$ always satisfies
		$\Delta(Z,\mathcal{O}_Z(1))=\mathcal{O}_Z(1)^{n}- h^0(Z,\mathcal{O}_Z(1))+n=\mathrm{deg}Z-(m+1)+n\ge 0$, where $\Delta$ is Fujita's $\Delta$-invariant (see \cite{Fuj75}). 	
		$Z$ is called a variety of minimal degree if $\mathrm{deg} Z = m-n+1$.  
		In this case, we have 
		\begin{align}\label{nonbir}
			K_X^2\ge L^2=\mathrm{deg}\phi_{|L|}\mathcal{O}_W(1)^2=\mathrm{deg}\phi_{|L|}\mathrm{deg}W\ge 2p_g-4.
		\end{align}
		
		Now assume $K_X^2=2p_g-4$  holds. We claim that $\phi_{|K_X|}$ is not birational. Otherwise, assume $\phi_{|K_X|}$ is birational for a contradiction. $K_X^2=2p_g-4$ implies that all the equalities in (\ref{bir}) hold. Hence, $p_g=3$,  $K_X^2=L^2$, $G=0$, $\pi=\epsilon=\mathrm{id}$ and $\phi_{|K_X|}$ is a birational morphism onto $\mathbb{P}^2$. Since $K_X$ is ample,  $\phi_{|K_X|}$ would be an isomorphism onto $\mathbb{P}^2$, a contradiction. 
		Therefore, $\mathrm{deg}\phi_{|K_X|}\ge 2$  and then all the equalities hold in (\ref{nonbir}). 
		By Lemma \ref{firstineq}, we have $G=0$, $\pi=\epsilon=\mathrm{id}$, $\mathrm{deg}\phi_{|K_X|}=2$ and $\mathrm{deg}W=p_g-2$. Thus, $K_X$ is Cartier, $|K_X|$ is base-point-free, $\phi_{|K_X|}$ is a double covering map, and $W$ is a surface of minimal degree in $\mathbb{P}^{p_g-1}$. 
	\end{proof}

	\section{stable surfaces with $K^2=2p_g-4$}
	
	Denote $N:=p_g-1$.
	It is well known that a surface $W$ of minimal degree $N-1$ in $\mathbb{P}^N$ is one of the following: (cf. \cite{Nag60})
	\begin{itemize}
		\item[(i)] $N=2$ and $W=\mathbb{P}^2$;
		\item[(ii)] $N=5$ and $W=\mathbb{P}^2$ embedded in $\mathbb{P}^5$ by $|\mathcal{O}_{\mathbb{P}^2}(2)|$;
		\item[(iii)] $N\ge 3$ and $W=\Sigma_d$, where $\Sigma_d$ is the Hirzebruch surface and $N-d-3$ is a non-negative even integer. $W$ is embedded in $\mathbb{P}^N$ by $|\Delta_0+\frac{N-1+d}{2}\Gamma|$, where $\Gamma$ is a fiber of ruling and $\Delta_0$ is the 0-section with self-intersection $-d$;
		\item[(iv)] $N\ge 3$ and $W\hookrightarrow \mathbb{P}^{N}$ is a cone over a rational curve of degree $N-1$ in $\mathbb{P}^{N-1}$.
	\end{itemize}
	
	We note in case (i)(ii)(iii), $W$ is smooth.  $\phi_{|K_X|}$ is a flat double covering map. The building data of $\phi_{|K_X|}$ is $(\mathcal{L},B)$, where $\mathcal{L}$ is a line bundle on $W$ and $B\in |2\mathcal{L}|$ is the branched curve (see \cite{AP13}).
	The discussion in \textsection \ref{prel} tells that $B$ is a reduced curve with at worst singularities of lc type.
	We give a detail description of $B$ in each case of $W$. It is similar to the situation of smooth surface of general type (cf. \cite{Hor76a}).
	\begin{itemize}
		\item[(i)] the case $N=2$ and $W=\mathbb{P}^2$: $B\in |\mathcal{O}_{\mathbb{P}^2}(8)|$. $X$ is a stable surface with $p_g=3$, $q=0$, and $K^2_X=2$. For general $B$, $X$ is smooth.
		\item[(ii)] the case $N=5$ and $W=\mathbb{P}^2$: $B\in |\mathcal{O}_{\mathbb{P}^2}(10)|$. $X$ is a stable surface with $p_g=6$, $q=0$, and $K^2_X=8$. For general $B$, $X$ is smooth.
		\item[(iii)] the case $W=\Sigma_d$ embedded in $\mathbb{P}^N$: $B\in |6\Delta_0+(N+3+3d)\Gamma|$, $N\ge2d-3$. $X$ is a stable surface with $p_g=N+1$, $q=0$, and $K^2_X=2N-2$. If $N\ge3d-3$, then for general $B$, $X$ is smooth. If $N<3d-3$, then generically $B=\Delta_0+B_0$ with $B_0$ irreducible and smooth, $X$ has canonical singularities lying over $\Delta_0\cap B_0$.
	\end{itemize}
	Conversely, given a branch curve $B$ on $W$ as described above, the associated double covering surface $X$ is a normal Gorenstein stable surface with $K^2_X=2p_g-4$.

	Now we consider case (iv), where $W$ is a cone $C_{N-1}$ over a rational curve $\mathbb{P}^1$ embedded in $\mathbb{P}^{N-1}$. $C_{N-1}$ has a minimal resolution $\sigma \colon \Sigma_{N-1} \to C_{N-1}$. 
	Let $V\subset |\mathcal{O}_{C_{N-1}}(1)|$ be the sublinear system consisting of hyperplane sections passing through the vertex $v$, which induces a rational map to $\mathbb{P}^1$. Moreover, $\sigma^*V=|(N-1)\gamma|+\Delta_0$ induces the ruling of $\Sigma_{N-1}$.
	The sublinear system $\phi_{|K_X|}^*V\subset |K_X|$ has base locus $\mathcal{D}\colon =\phi_{|K_X|}^{-1}(v)$, which is zero-dimensional since  $\phi_{|K_X|}$ is a finite morphism.
	Let $\mu\colon \widehat{X} \to X$ be a composition of resolution of singularities and blowup of smooth points in the base locus of pullbacks of  $\phi_{|K_X|}^*V$, such that $\mu^*\phi_{|K_X|}^*V=|L|+G$ with $|L|$ base-point-free, where  $|L|$ is the movable part and $G$ is the fixed part. We see that $L\sim (N-1)F$, where $F$ is the preimage of a line $\ell\in C_{N-1}$ passing through $v$.
	We see that $\phi_{|K_X|}{}_{\circ}\mu$ factors through  $\phi_{|L|}$ and then factors through a morphism $\gamma:\widehat{X}\to \Sigma_{N-1}$. 
	We have the following commutative diagram:
	\[
	\xymatrix{
		&\widehat{X}\ar[r]^{\mu}\ar[d]^{\gamma}\ar@/_5pc/[ddr]_{\phi_{|L|}} &X \ar[d]^{\phi_{|K_X|}}\\
		&\Sigma_{N-1}\ar[r]^{\sigma}\ar[dr]  & C_{N-1}\ar@{-->}[d]^{\phi_{V}}	\\
		&&\mathbb{P}^1}
	\]\\

	We have 
	$$\mu^*K_X\sim L+G \sim (N-1)F+G$$
	and 
	$$\mu^*K_X=K_{\widehat{X}}+\sum Z_i^e-\sum \mathcal{E}_k,$$
	where $\sum Z_i^e$ denotes the sum of pullbacks of numerical fundamental cycles associated to elliptic Gorenstein singularities in $\mathcal{D}$ and $\sum \mathcal{E}_k$ denotes the sum of pullbacks of $(-1)$ curves associated to the smooth points in $\mathcal{D}$. 
	By \textsection \ref{fundmtclcle} we have $$G\ge \sum Z_i^e+\sum Z_j^c+\sum \mathcal{E}_k$$
	and 
	$$ LZ_i^e,LZ_j^c,L \mathcal{E}_k>0,$$
	where $\sum Z_j^c$ denotes the sum of pullbacks of the numerical fundamental cycles of $ADE$-singularities in $\mathcal{D}$. 
	Since $L\sim (N-1)F$, we have $FZ_i^e,FZ_j^c,F \mathcal{E}_k>0$. Moreover, $FZ_i^e\ge 2$  as the degree of the morphism $Z_i^e\hookrightarrow \widehat{X}\stackrel{\phi_{|F|}}{\longrightarrow} \mathbb{P}^1$ must be greater than $1$.
	
	Therefore, we have 
	\begin{align}\label{baseparteq}
		\begin{split}
			F\mu^*K_X &=F\mu^*\phi_{|K_X|}^*\mathcal{O}_{C_{N-1}}(1)=2\ell \cdot \mathcal{O}_{C_{N-1}}(1)=2\\
			&=F(L+G)=FG\\
			&\ge F(\sum Z_i^e+\sum Z_j^c+\sum \mathcal{E}_k).
		\end{split}
	\end{align}
	
	Hence we have 
	\begin{itemize}
		\item[a)] either $\sum Z_i^e= Z , FZ =2$ and $\sum Z_j^c=\sum \mathcal{E}_k=0$, which means $\mu$ is a resolution of an elliptic Gorenstein singularity $p$ and $\mathcal{D}=\{p\}$;
		\item[b)] or $\sum Z_i^e=0$, which means $\mathcal{D}$ contains no elliptic Gorenstein singularity.
	\end{itemize}
	
	For case a), we have $\mu^*K_{\widetilde{X}}=K_{\widehat{X}}+Z $ and $g(F)=1+\frac{FK_{\widehat{X}}}{2}=1+\frac{2-FZ}{2}=1$. The branched divisor of $\gamma$ is
	
	\begin{align*}
		B &=\gamma_*(K_{\widehat{X}}-\gamma^*K_{\Sigma_{N-1}})\\
		&=\gamma_*(\mu^*K_{\widetilde{X}}-Z -\gamma^*K_{\Sigma_{N-1}})\\
		&\sim\gamma_*(\gamma^*\sigma^*\mathcal{O}_{C_{N-1}}(1)-Z -\gamma^*K_{\Sigma_{N-1}})\\
		&=2\sigma^*\mathcal{O}_{C_{N-1}}(1)-2\Delta_0-2K_{\Sigma_{N-1}}\\
		&\sim 4\Delta_0+4N\Gamma.
	\end{align*} 
	Since $FZ =2$, $B\not \supseteq \Delta_0$. Moreover, $B$ has at worst singularities of lc type.  Since $\mathrm{supp}(\gamma^*\Delta_0)$ is either a simple elliptic curve or a cycle of rational curves and $B\Delta_0=4$, $B\cap \Delta_0$ consists of simple points or $A_n$ singularities of $B$ and $\Delta_0$ is not tangent to $B$. If $B\cap \Delta_0$ consists of only simple points, $p$ is a simple elliptic singularity. Otherwise, $p$ is a cusp elliptic singularity.
	
	For case b), we have
	$\mu^*K_{\widetilde{X}}=K_{\widehat{X}}-\sum E_j$ and $g(F)=1+\frac{FK_{\widehat{X}}}{2}=1+\frac{2+F\sum E_j}{2}\le 2$ by \ref{baseparteq}. 
	We show that $F\sum E_j\not=2$. Otherwise, either $\mu$ blows up two different points $p_1,p_2\in \mathcal{D}$, or a point $p_1\in \mathcal{D}$ and a point $p_2$ infinitely near to $p_1$. 
	The case $p_1\not=p_2\in \mathcal{D}$ is impossible. Since then $\gamma^*\Delta_0=E_1+ E_2$, where $E_i$ is the $(-1)$-curve corresponding to $p_i$, and  $-2=(\sum E_j)^2=(\gamma^*\Delta_0)^2=-2(N-1)\le -4$, a contradiction. 
	The case  $p_2$ is infinitely near to $p_1$ is also impossible. 
	Let $E_2$ be the $(-1)$-curve corresponding to $p_2$ and $\overline{E_1}$ be the strict transformation of the $(-1)$-curve corresponding to $p_1$. $E_1=\overline{E_1}+E_2$ is the pullback of $(-1)$-curve corresponding to $p_1$. Since $\overline{E_1}$ is $\gamma$-exceptional, $\gamma^*\Delta_0\cdot \overline{E_1}=0$. 
	This together with $FE_2=1$ implies $\gamma^*\Delta_0=2E_2+ \overline{E_1}$. Therefore,
	$-2=(2E_2+ \overline{E_1})^2=(\gamma^*\Delta_0)^2=-2(N-1)\le -4$, a contradiction.

	Therefore, $F\sum E_j=0$, which means $\sum E_j=0$ and $\mu$ only resolves $ADE$-singularties.
	$\mu^*K_{\widetilde{X}}=K_{\widehat{X}}$. 
	$g(F)=1+\frac{FK_{\widehat{X}}}{2}=1+\frac{2}{2}=2$. The branched divisor of $\gamma$ is
	
	\begin{align*}
		B &=\gamma_*(K_{\widehat{X}}-\gamma^*K_{\Sigma_{N-1}})\\
		&=\gamma_*(\mu^*K_{\widetilde{X}}-\gamma^*K_{\Sigma_{N-1}})\\
		&\sim\gamma_*(\gamma^*\sigma^*\mathcal{O}_{C_{N-1}}(1)-\gamma^*K_{\Sigma_{N-1}})\\
		&=2\sigma^*\mathcal{O}_{C_{N-1}}(1)-2K_{\Sigma_{N-1}}\\
		&\sim 6\Delta_0+4N\Gamma.
	\end{align*} 
	
	If $\Delta_0\not\subseteq B$, $\Delta_0 B=-6(N-1)+4N\ge 0$ implies $N=3$ and $\Delta_0 \cap B=\emptyset$. $\mathcal{D}$ consists of two $A_1$ singularities.  If $\Delta_0\subseteq B$, then $B=B_0+\Delta_0$ and $\Delta_0 B_0=-5(N-1)+4N=5-N\ge 0$. Hence $N=3,4,5$. If $N=5$, $\mathcal{D}$ consists of one $A_1$ singularity. If $N=4$, $\mathcal{D}$ consists of one $A_2$ singularity. If $N=3$ and  $\Delta_0\cap B_0$ consists of two nodes of $B$, $\mathcal{D}$ consists of a singularity of type $A_3$.
	If $N=3$ and  $\Delta_0$ is tangent to $B_0$, $\mathcal{D}$ consists of a singularity of type $D_4$.
	If $N=3$,  $\Delta_0\cap B_0$ consists of one point and $\Delta_0$ is not tangent to $B_0$, then $B_0$ locally has an $A_n$ singularity at $\Delta_0\cap B_0$ and $\Delta_0\cap B_0$ is a singularity of $B$ of type $D_{n+3}$.  
	$\mathcal{D}$ consists of a singularity of type $D_{n+4}$. 
	Since $(\Sigma_2,B_0)$ is birational to $(\mathbb{P}^2,\widetilde{B_0}\in |10|)$, $\widetilde{B_0}$ could have a singularity $p$ of $A_n$ type for $n\le 9$. Hence $n\le 9$.  

	Therefore, $\mathcal{D}$ consists of a singularity of type $A_3$, $D_4$,$D_{n+4}$ ($0\le n\le 9$).
	
	\smallskip
	We summarize results in this section in the following theorem.
	\begin{theorem}
		Let $X$ be a  normal stable surface with $|K_X|$ not composed with a pencil and $K_X^2=2p_g-4$ and $W$ be its canonical image.
		Then if $W$ is smooth, $X$ is determined by its double covering data $(W,B)$ which is one of the following:
		\begin{itemize}
			\item[(i)] $W=\mathbb{P}^2$, $B\in |\mathcal{O}_{\mathbb{P}^2}(8)|$; ($N=2$)
			\item[(ii)] $W=\mathbb{P}^2$: $B\in |\mathcal{O}_{\mathbb{P}^2}(10)|$; ($N=5$)
			\item[(iii)] $W=\Sigma_d$, $B\in |6\Delta_0+(N+3+3d)\Gamma|$. ($N\ge2d-3$ and $N-d-3$ is a non-negative even integer)
		\end{itemize}
		
		If $W$ is singular, $X$ is obtained by taking a double cover with covering data $(\Sigma_d,B)$ then contracting the preimage of $\Delta_0$. All possible $(\Sigma_d,B)$ are listed in Table \ref{covering table}.
		
		In either case $B$ is a reduced curve with at worst singularities of lc type.
	\end{theorem}

	\begin{table}
		\resizebox{\textwidth}{!}{   
			\begin{tabular}{|c|c|c|c|c|}
				\hline
				$\Sigma_{N-1}$ & $B$ & $\Delta_0\subset B$ ?  & \tabincell{c}{relation of \\$\Delta_0$ and $B$}&\tabincell{c}{ singularity type\\ of $\Psi^{-1}(p)$}\\
				\hline
				$\Sigma_2$& $|6\Delta_0+12\Gamma|$ & No & $B\cdot\Delta_0=0$ & two $A_1$  \\
				\hline
				$\Sigma_2$& $|6\Delta_0+12\Gamma|$& Yes & \tabincell{c}{ $(B-\Delta_0)\cap\Delta_0$ consists of\\ either two nodes, \\ an $A_2$ singularity,\\ or an $D_{n+3}$  singularity on $B$, $0\le n\le 9$} & $A_3$,  $D_{4}$ or $D_{n+4}$, $0\le n\le 9$ \\
				\hline
				$\Sigma_3$& $|6\Delta_0+16\Gamma|$ & Yes & $(B-\Delta_0)\cdot\Delta_0=1$ & $A_2 $\\
				\hline
				$\Sigma_4$& $|6\Delta_0+20\Gamma|$& Yes & $(B-\Delta_0)\cdot\Delta_0=0$ & $A_1$ \\
				\hline
				\tabincell{c}{$\Sigma_{N-1}$,\\ $N\ge 3$} & $|4\Delta_0+4N\Gamma|$ & No & \tabincell{c}{$B\cap\Delta_0$ consists of\\ 4 distinct points}& simple elliptic singularity\\
				\hline
				\tabincell{c}{$\Sigma_{N-1}$,\\ $N\ge 3$} & $|4\Delta_0+4N\Gamma|$ & No & \tabincell{c}{$B\cap\Delta_0$ consists of\\ two at worst $A_{n}$ singularities on $B$}& cusp elliptic singularity\\
				\hline
		\end{tabular}}
		\caption{classifications of surfaces in case (iv)}
		\label{covering table}
	\end{table}
	
	\begin{remark}
		Stable Horikawa surfaces with canonical image $W$ smooth  can be deformed into a canonical surface by deforming the branched curve $B$. Hence such stable surfaces are degenerations of canonical surfaces.
		Stable Horikawa  surfaces of case (i) have been studied in \cite{ant19}.
		Stable Horikawa  surfaces of case (ii) (iii) have been studied in \cite{RR22}.
	\end{remark}
	
	\begin{remark}
		When the canonical image $W$ of a stable Horikawa surface is singular (i.e. $W$ is a cone $C_{N-1}$), the vertex will produce an elliptic Gorenstein singularity of degree $2(N-1)$ on $X$, where $N=p_g-1$. Those stable Horikawa surfaces with cusp elliptic singularities as the preimage of the vertex are degenerations of those with simple elliptic singularities as the primage of the vertex.
		It is known that a simple elliptic singularity of degree $\ge 10$ is not smoothable (see \cite{pink74} {Chapter 7}). Hence the stable surfaces with such elliptic singularities form a distinguished irreducible component different from that of canonical surfaces in the moduli. Moreover, the number of moduli is $\mathrm{dim}|B|-\mathrm{dim} \mathrm{Aut}\Sigma_{N-1}=10N+14-(N-1+5)=9N+10=9p_g+1$.  
	\end{remark}

	\section{stable surfaces with $2p_g-4<K_X^2<2p_g-3$}
	
	\begin{lemma}\label{simplereduction}
		Let $X$ be a Gorenstein (i.e $K_X$ Cartier) lc surface, and $\ell$ be a reduced and irreducible curve on it. 
		
		If $ K_X \ell<0$ and $\ell^2<0$, then $\ell$ passes no elliptic Gorenstein singularity of $X$.
		%
	\end{lemma}
	\begin{proof}
		Let $\epsilon \colon Y\to X$ be a resolution of singularities. Then $\epsilon^*K_X=K_Y+\sum Z_i$, where $Z_i$'s are the elliptic cycle associated with the elliptic Gorenstein singularities. 
		$\epsilon^*\ell=\bar{\ell}+G$ where $G$ is an effective $\mathbb{Q}$-divisor supported on the exceptional locus. 
		We have 
		\begin{align*}
			0>&\ell K_X=\bar{\ell}\epsilon^*K_X=\bar{\ell}K_Y+\bar{\ell}\sum Z_i,\\
			0>&\ell^2=\bar{\ell}\epsilon^*\ell=\bar{\ell}(\bar{\ell}+G).  
		\end{align*}
		Hence we have
		\begin{align*}
			0\le p_a(\bar{\ell})&=1+\frac{\bar{\ell}^2+\bar{\ell}K_Y}{2}\\
			&=1+\frac{{\ell}^2-\bar{\ell}G+\ell K_X-\bar{\ell}\sum Z_i}{2}\\
			&< 1+\frac{-\bar{\ell}G-1-\bar{\ell}\sum Z_i}{2}.
		\end{align*}
		Hence $p_a(\bar{\ell})=0$, $\bar{\ell}\sum Z_i=0$. Therefore $\ell$ passes no elliptic Gorenstein singularity.
	\end{proof}
	
	\begin{lemma}\label{blowdown}
		Let $X$ be a Gorenstein normal lc surface. We assume that
		\begin{itemize}
			\item [-]	$K_{X}=L+G$, where $G$ is an effective Cartier divisor and $L$ is a Cartier divisor such that $LG=0$;
			\item [-] $G$ passes no $ADE$ singularities;
			\item [-] the intersection matrix of the support of $G$ is negative definite.   
		\end{itemize}
		Then $G$ is contained in the smooth locus of $X$. Moreover, $G$ is a sum of pullbacks of (-1)-curves, i.e. there is $\eta\colon X\to Y$ which is a composition of blowups $\eta_i$, $i=1,...,m$ such that $ G=\sum\limits^m_{i=1} \mathcal{E}_i$ and $ L=\eta^*K_Y$, where $\mathcal{E}_i$ is the pullback of $(-1)$ curve of $\eta_i$. 
		
	\end{lemma}
	\begin{proof}
		Since the intersection matrix of the support of $G$ is negative definite,  $GK_{X}=G^2<0$. Write $G=\sum b_iC_i$. There exists  $C_i$ such that $C_iK_{X}<0$. By Lemma \ref{simplereduction}, $C_i$ passes no elliptic Gorenstein singularities. By assumption $C_i$ passes no $ADE$ singularities and then $C_i$ is contained in the smooth locus of $X$. Therefore, $C_i$ is a (-1) curve. Let $\eta_1\colon X \to Y_1$ be the blowdown of $C_i$. We have $K_{Y_1}=\eta_1( L)+\eta_1(G)$. Inductively we obtain blowdowns $\eta_j\colon Y_{j-1}\to Y_{j}$, $j=2,..., m$ and finally $\eta_1{}_{\circ}...{}_{\circ}\eta_m (G)$ is zero-dimensional. We have  $K_{Y_m}=\eta_1{}_{\circ}...{}_{\circ}\eta_m ( L)$. Let $Y=Y_m$ and  $\eta=\eta_1{}_{\circ}...{}_{\circ}\eta_m$. We have $ L=\eta^*K_Y$ and $G=\sum\limits^k \mathcal{E}_i$, where $\mathcal{E}_i$ is the pullback of $(-1)$ curve of $\eta_i$.  
	\end{proof}

	\begin{lemma}
		Let $X$ be a normal stable surface with $2p_g-4<K_X^2< 2p_g-3$. Then $\phi_{|K_{X}|}$ is not birational.
	\end{lemma}
	\begin{proof}
		Assume $\phi_{|K_{X}|}$ is birational for a contradiction. By Lemma \ref{castelnuovo}, we have
		$K_X^2\ge L^2\ge 3p_g-7=2p_g-4+p_g-3$. 
		As $2p_g-4<K_X^2< 2p_g-3$, we have $p_g=3$, $L^2=2$ and $3>K_X^2>L^2=2$. 
		Since $K_X^2=L(L+G-\sum a_E\epsilon^*E)+G\epsilon^*\pi^*K_X$, we have $LG=0$. Hence $\epsilon=\mathrm{id}$, $K_{\widetilde{X}}=|L|+G$ and $LG=0$. Let 
		$\mu\colon  \widehat{X}\to \widetilde{X}$ be a resolution of $ADE$ singularities on $G$. 
		We have $K_{\widehat{X}}=\mu^*K_{\widetilde{X}}=|\mu^*L|+\mu^*G$. By Lemma \ref{blowdown}, there exists
		$\eta\colon \widehat{X}\to Y$ which is a composition of blowups, such that $\mu^*G$ is a sum of pullback of $(-1)$ curves and $K_Y=\eta(\mu^*L)$. We see that $K_Y$ is Cartier and base-point-free, and $\phi_{|K_{Y}|}$ is birational onto $\mathbb{P}^2$. 
		
		Let $\gamma\colon \widetilde{Y}\to Y$ be a minimal resolution and $\theta=\phi_{|K_{Y}|}{}_{\circ}\gamma$. 
		\[
		\xymatrix@C=3em@R=10ex{
			X &\widetilde{X}\ar[l]_{\pi}\ar[dr]_{\phi_{|L|}}&\widehat{X}\ar[l]_{\mu}\ar[d]_(0.3){\phi_{|\mu^*L|}}\ar[r]^{\eta}&	Y \ar[dl]^(0.3){\phi_{|K_{Y}|}}  & \widetilde{Y}\ar[l]_{\gamma}\ar@/^2pc/[dll]^{\theta}  \\
			&&\mathbb{P}^{2}	
		}
		\]
		
		If $\theta=\mathrm{id}$, then $\phi_{|K_{Y}|}=\mathrm{id}$ and $K_Y=K_{\mathbb{P}^2}=\mathcal{O}_{\mathbb{P}^2}(-3)$, which contradicts the fact that $\phi_{|K_{Y}|}$ is birational. 
		Therefore, $\theta$ is a succession of blowups. 
		Note that $\widetilde{X}$ has at worst $ADE$-singularities and elliptic Gorenstein singularities, and so are $\widehat{X}$ and $Y$. Moreover, the support of the excetptional locus of $\gamma$ denoted by $\mathrm{Supp}\,\mathrm{Exc}(\gamma)$ is contained in $\mathrm{Supp}\,\mathrm{Exc}(\theta)$ which consists of several trees of rational curves. Therefore, there is no elliptic Gorenstein singularity on $Y$ and	$\gamma$ resolves at worst canonical singularities. Hence $\gamma^*K_{Y}=K_{\widetilde{Y}}$. 
		Since $\theta\not=\mathrm{id}$, there is a (-1) curve $\mathcal{E}$ on $\widetilde{Y}$. We have $-1=\mathcal{E}K_{\widetilde{Y}}=\mathcal{E}\gamma^*K_{Y}=\gamma(\mathcal{E})K_{Y}$, which contradicts to the fact that $|K_Y|$ is base-point-free. 
		Therefore, $\phi_{|K_{X}|}$ is not birational.
		
	\end{proof}
	\begin{theorem}\label{ncwp}
		Let $X$ be a normal stable surface with $2p_g-4<K_X^2< 2p_g-3$ and $|K_X|$ is not composed with a pencil. 
		Then $\phi_{|K_{X}|}$ is of degree two and $K_X^2\ge  2p_g-4+\frac{1}{3}$. 
		
		If the equality holds, $X$ contains a unique $\frac{1}{3}(1,1)$-singularity $p$, whose exceptional curve $E$ is a (-3) curve. 
		Let $\pi \colon \widetilde{X}\to X$ be the resolution of $p$. We have $\pi^*K_X=K_{\widetilde{X}}+\frac{1}{3}E$. $|K_{\widetilde{X}}|$ is base-point-free and $\phi_{|K_{\widetilde{X}}|}$ is a generically finite morphism of degree 2 onto $W$ which is a surface of minimal degree.
		Let $B$ be the branched curve. We have $B\supseteq \phi_{|K_{\widetilde{X}}|}(E)$. 
		Moreover, $(W,\phi_{|K_{\widetilde{X}}|}(E))$ is one of the following:
		\begin{itemize}
			\item[a)] $(\mathbb{P}^2, \ell_0)$, where $\ell_0$ is a line. ($B\in |8\ell|$);
			\item[b)] $(\Sigma_d, \Gamma_0)$, where $\Gamma_0$ is a fiber of the ruling. ($B\in |6\Delta_0+(N+3d+3)\Gamma|$); ($N=p_g-1$)
			\item[c)] $(\Sigma_d, \Delta_0)$, where $\Delta_0$ is the zero section of $\Sigma_d$. ($B\in |6\Delta_0+(4d+6)\Gamma|$); ($N=d+3$) 
			\item[d)] $(C_{N-1},\ell_0)$, where $\ell_0$ is a line passing through the vertex of the cone. ($N\ge 3$)
		\end{itemize}
		
	\end{theorem}
	\begin{proof}
		By the previous lemma, $\phi_{|K_{X}|}$ is not birational. Hence, we have 
		\begin{align*}
			2p_g-3> K_X^2=&(L+G)\epsilon^*\pi^*K_X\\
			=&L(L+G-\sum a_EE)+G\epsilon^*\pi^*K_X\\
			\ge & L^2+LG\ge L^2\\
			=&\mathrm{deg}\phi_{|L|}\mathrm{deg}W\\
			\ge &2(p_g-2).
		\end{align*}
		
		
		The assumption $2p_g-4<K_X^2< 2p_g-3$ implies $L^2=2(p_g-2)$, $LG=0$, and $0<-\sum a_EEL+G\epsilon^*\pi^*K_X<1$. Hence $\epsilon=\mathrm{id}$ by Lemma 3.1, $\mathrm{deg}\phi_{|L|}=2$ and $W$ is a surface of minimal degree. 
		Moreover, 
		we claim that $-\sum a_EE\not=0$, i.e. $\pi\not=\mathrm{id}$. Otherwise,  $K_X$ is Cartier and then $GK_X$ will be an integer, which contradicts to $0<GK_X<1$. 
		
		Let 
		$\mu\colon  \widehat{X}\to \widetilde{X}$ be a resolution of $ADE$ singularities on $G$. 
		We have $K_{\widehat{X}}=\mu^*K_{\widetilde{X}}=|\mu^*L|+\mu^*G$. By Lemma \ref{blowdown}, there exists
		$\eta\colon \widehat{X}\to Y$ which is a composition of blowups, such that $\mu^*G$ is a sum of pullback of $(-1)$ curves and $K_Y=\eta_*(\mu^*L)$.
		We introduce $\theta\colon=\pi_{\circ}\mu$ to  rewrite $\theta^*K_X=K_{\widehat{X}}-\sum a_{\widehat{E}} \widehat{E}=\widehat{L}+\widehat{G}-\sum a_{\widehat{E}}\widehat{E}$ where $\widehat{L}=\mu^*L$, $\widehat{G}=\mu^*G$ and $\widehat{E}=\mu^*E=\mu^{-1} (E)$.  We have the following commutative diagram:
		\[
		\xymatrix{
			\widehat{X}\ar[r]^{\mu} \ar[d]_{\eta} \ar@(ur,ul)[rr]^{\theta}& \widetilde{X}\ar[r]^{\pi}\ar[d]_{\varphi_{|L|}}\ar[dr]^{\phi_{|L|}} & X\ar@{-->}[d]^{\phi_{|K_X|}}\\
			Y\ar@/_2pc/[rr]_{\phi_{|K_{Y}|}}\ar[r]^{\varphi_{|K_{Y}|}} &W\ar@{^(->}[r]  &\mathbb{P}^{p_g-1}	
		}
		\]
		
		We show that $-\sum a_{\widehat{E}}\widehat{E}\widehat{L}\not =0$. 
		Otherwise, $\widehat{E}\widehat{L} =\widehat{E}\eta^*K_Y=\eta(\widehat{E})K_Y=0$ for each $\widehat{E}\in \mathrm{supp}(-\sum a_{\widehat{E}}\widehat{E})$. Hence $\eta(\widehat{E})$ is either a point or a (-2) curve on $Y$. 
		Therefore, $\mathrm{supp}(-\sum a_{\widehat{E}}\eta(\widehat{E}))$ either consists of (-2) curves or is zero dimensional. 
		Since $K_Y$ is nef and big, a sum of (-2) curves on $Y$ is always contractible (i.e., the intersection matrix is negative definite). 
		As a result, $\mathrm{supp}(\hat{G}-\sum a_{\widehat{E}}\widehat{E})$ is contractible. Therefore, we have $(\hat{G}-\sum a_{\widehat{E}}\widehat{E})^2\le 0$. 
		Since $\hat{G}-\sum a_{\widehat{E}}\widehat{E}\ge 0$ and $\theta^*K_X$ is nef,  we have $\theta^*K_X(\hat{G}-\sum a_{\widehat{E}}\widehat{E})\ge 0$. 
		As
		$\theta^*K_X=\hat{L}+\hat{G}-\sum a_{\widehat{E}}\widehat{E}$, 
		we have 
		\begin{align*}
			0\le \theta^*K_X(\hat{G}-\sum a_{\widehat{E}}\widehat{E})=&(\hat{L}+\hat{G}-\sum a_{\widehat{E}}\widehat{E})(\hat{G}-\sum a_{\widehat{E}}\widehat{E})\\
			=&(\hat{G}-\sum a_{\widehat{E}}\widehat{E})^2\le 0,
		\end{align*}
		As a result, we have $\hat{G}=-\sum a_{\widehat{E}}\widehat{E}=0$, a contradiction to $-\sum a_EE\not=0$.

		Now assume $\widehat{E}'\widehat{L}>0$ for some $\widehat{E}'$. Then $\widehat{E}'\widehat{L}=\widehat{E}'\eta^*K_Y=\eta(\widehat{E}')K_Y>0$, hence $\eta(\widehat{E}')$ is not a point. We see that $ \widehat{E}'\widehat{G}=\widehat{E}' \sum \mathcal{E}_k\ge 0$ and $\widehat{E}'K_{\widetilde{X}}=\widehat{E}'\widehat{L}+\widehat{E}'\widehat{G}\ge \widehat{E}'\widehat{L}>0$. Hence $\widehat{E}'^2\le -3$. 
		By Remark \ref{discrep} we have $$-a_{\widehat{E}'}=\frac{-2-\widehat{E}'^2-\sum\limits_{\widehat{E}\not=\widehat{E}'} a_{\widehat{E}}\widehat{E}\widehat{E}'}{-\widehat{E}'^2}\ge \frac{-2-\widehat{E}'^2}{-\widehat{E}'^2}\ge \frac{1}{3}.$$
		'=' holds if and only if $\widehat{E}'$ is an (-3)  curve, $\sum\limits_{\widehat{E}\not=\widehat{E}'} \widehat{E}\widehat{E}'=0$,  $\widehat{E}'\widehat{L}=1$, and $\widehat{E}'\widehat{G}=0$. 
		
		Therefore, we have 
		\begin{align*}
			K_X^2=&(\widehat{L}+\widehat{G}) \theta^*K_X\\
			=&\widehat{L}^2-\sum a_{\widehat{E}}\widehat{E}\widehat{L}+\widehat{G} \theta^*K_X\\
			\ge & 2(p_g-2)+\frac{1}{3}.
		\end{align*}
		
		Moreover, if '=' holds, $\widehat{L}^2=L^2=2(p_g-2)$, $-\sum a_{\widehat{E}}\widehat{E}\widehat{L}=\frac{1}{3}$ and $\widehat{G} \theta^*K_X=0$. 
		Since $G\pi^*K_X=\widehat{G} \theta^*K_X=0$ and $LG=0$, we see that $G=0$ by Lemma \ref{firstineq}. 
		Therefore, we see that $\mu=\eta=\mathrm{id}$ and $\widehat{L}=K_{\widetilde{X}}$. 
		On the other hand, $-\sum a_{\widehat{E}}\widehat{E}\widehat{L}=\frac{1}{3}$ implies $-\sum\limits_{\widehat{E}\not=\widehat{E}'} a_{\widehat{E}}\widehat{E}\widehat{L}=0$, $\widehat{E}'\widehat{L}=1$ and $- a_{\widehat{E}'}=\frac{1}{3}$, which implies
		$\widehat{E}'$ is an (-3)  curve, $\sum\limits_{\widehat{E}\not=\widehat{E}'} \widehat{E}\widehat{E}'=0$ and $\widehat{E}'\widehat{L}=1$. However, $-\sum\limits_{\widehat{E}\not=\widehat{E}'} a_{\widehat{E}}\widehat{E}\widehat{L}=-\sum\limits_{E\not=E'} a_{E}EK_{\widetilde{X}} =0$ and $\sum\limits_{\widehat{E}\not=\widehat{E}'} \widehat{E}\widehat{E}'=\sum\limits_{E\not=E'} EE'=0$ implies $$0=-\sum\limits_{E\not=E'} a_{E}E\cdot \pi^*K_X=(-\sum\limits_{E\not=E'} a_{E}E)^2$$ and then $-\sum\limits_{E\not=E'} a_{E}E=0$.
		Therefore, $-\sum a_{E}E=-\frac{1}{3}E'$, where $E'$ is a (-3) curve. 
		Moreover, $\widehat{L}^2=L^2=2(p_g-2)$ implies
		$\phi_{|K_{\widetilde{X}}|}$ is a generically finite morphism of degree 2 onto $W$ which is a surface of minimal degree. 
		Since
		$1=EK_{\widetilde{X}}=E\cdot\phi^*_{|K_{\tilde{X}}|} \mathcal{O}(1)=\phi_{|K_{\tilde{X}}|}{}_*E\cdot \mathcal{O}(1)=\mathrm{deg}\phi_{|K_{\tilde{X}}|}|_{E}\cdot \phi_{|K_{\tilde{X}}|}(E)\cdot \mathcal{O}(1)$, we see that $\phi_{|K_{\widetilde{X}}|}(E)\subseteq B$ and $\phi_{|K_{\tilde{X}}|}(E)$ is a line in $W$.
		Therefore,
		$(W,\phi_{|K_{\widetilde{X}}|}(E),B)$ is one of the following:
		\begin{itemize}
			\item[a)] $(\mathbb{P}^2, \ell_0)$, where $\ell_0$ is a line. $B\in |8\ell|$;
			\item[b)] $(\Sigma_d, \Gamma_0)$, where $\Gamma_0$ is a fiber of the ruling. $B\in |6\Delta_0+(N+3d+3)\Gamma|$; ($N=p_g-1$)
			\item[c)] $(\Sigma_d, \Delta_0)$, where $\Delta_0$ is the zero section of $\Sigma_d$. $B\in |6\Delta_0+(4d+6)\Gamma|$; ($N=d+3$) 
			\item[d)] $(C_{N-1},\ell_0)$, where $\ell_0$ is a line passing through the vertex of  the cone. ($N\ge 3$)
		\end{itemize}
	\end{proof}

	\begin{example}
		Let $B=\Gamma_0+B_0\in |6\Delta_0+(N+3d+3)\Gamma|$ be a nodal curve on $\Sigma_d$. Let $\bar{X}$ be a double covering surface over $\Sigma_d$ with $B$ as the branched curve and let $D$ be the preimage of $\Gamma_0$. There are 6 $A_1$ singularities on $D$ lying over $\Gamma_0\cap B_0$. Let $\epsilon\colon \widetilde{X}\to \bar{X}$ be a minimal resolution of these  6 $A_1$ singularities and $\bar{D}$ be the proper transformation of $D$. Then $\bar{D}$ is a (-3) curve. Contracting $\bar{D}$ we obtain a KSBA stable surface $X$ with $K_X^2=2p_g-4+\frac{1}{3}$ of type (b). 
	\end{example}
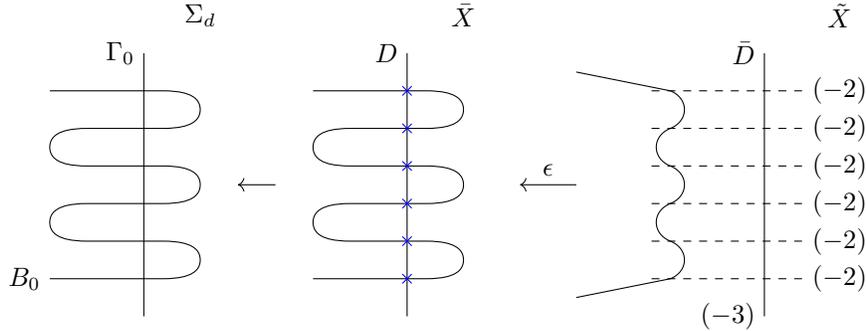
	\begin{tikzpicture}
		\begin{scope}[scale=0.5] 
			\draw (0,0) node[left]{$B_0$} to (3,0) to[in=270,out=0] (4,0.5) to[in=0,out=90] 
			(3,1) to (1,1) to[in=270,out=180] (0,1.5) to[in=180,out=90] (1,2) to (3,2) to[in=270,out=0] (4,2.5) to[in=0,out=90]  
			(3,3) to (1,3) to[in=270,out=180] (0,3.5) to[in=180,out=90] (1,4) to (3,4) to[in=270,out=0] (4,4.5) to[in=0,out=90] 
			(3,5) to (0,5);
			\draw (2.5,-1)  to (2.5,6) node[left]{$\Gamma_0$}; 
			\node at (4,7) {$\Sigma_d$};
			\draw[->] (6,2.5) to (5,2.5);
		\end{scope}
		\begin{scope}[scale=0.5,xshift=7cm] 
			\draw (0,0) to (3,0) to[in=270,out=0] (4,0.5) to[in=0,out=90] 
			(3,1) to (1,1) to[in=270,out=180] (0,1.5) to[in=180,out=90] (1,2) to (3,2) to[in=270,out=0] (4,2.5) to[in=0,out=90]  
			(3,3) to (1,3) to[in=270,out=180] (0,3.5) to[in=180,out=90] (1,4) to (3,4) to[in=270,out=0] (4,4.5) to[in=0,out=90] 
			(3,5) to (0,5);
			\draw (2.5,-1) to (2.5,6) node[left]{$D$}; 
			\foreach \y in {(2.5,0),(2.5,1),(2.5,2),(2.5,3),(2.5,4),(2.5,5)}
			\draw[blue] plot[mark=x, mark size=5] coordinates{\y};
			\node at (4,7) {$\bar{X}$};
			\draw[->] (7,2.5) to node[above]{$\epsilon$} (5.5,2.5);
		\end{scope}
		\begin{scope}[scale=0.5,xshift=14cm] 
			\draw (0,-0.5) to (2.5,0) .. controls (3,0.2) and (3,0.8) ..
			(2.5,1)  .. controls (2,1.2) and (2,1.8) .. (2.5,2)
			(2.5,2)  .. controls (3,2.2) and (3,2.8) .. (2.5,3)
			.. controls (2,3.2) and (2,3.8) .. (2.5,4)
			.. controls (3,4.2) and (3,4.8) .. (2.5,5)
			(2.5,5) to (0,5.5);
			\draw (5,-1) node[left]{$(-3)$} to (5,6) node[left]{$\bar{D}$}; 
			\foreach \y in {0,1,2,3,4,5}
			{\draw[dashed] (2,\y) to (6,\y);
				\node at (7,\y) {$(-2)$};}
			\node at (7,7) {$\tilde{X}$};
		\end{scope}
	\end{tikzpicture}
	\begin{example}
		Let $B=\Gamma_0+B_0\in |4\Delta_0+4N\Gamma|$ be a nodal curve on $\Sigma_{N-1}$, $N\ge 3$, where $\Gamma_0$ is a fiber of ruling. Assume that $\Delta_0\cap B$ consists of 4 different points for simplicity. Let $\bar{X}$ be a double covering surface over $\Sigma_{N-1}$ with $B$ as the branched curve and let $D,Z$ be the preimage of $\Gamma_0, \Delta_0$. There are 4 $A_1$ singularities on $D$ lying over $\Gamma_0\cap B_0$ and $DZ=1$. Let $\epsilon\colon \widetilde{X}\to \bar{X}$ be a composition of a minimal resolution of these  4 $A_1$ singularities and a blowup of $Z\cap D$, and $\bar{D},\bar{Z}$ be the proper transformation of $D,Z$. Then $\bar{D}$ is a (-3) curve and $\bar{Z}$ is an elliptic curve with $\bar{Z}^2=-(2N-1)$. Contracting $\bar{D},\bar{Z}$ we obtain a KSBA stable surface $X$ with $K_X^2=2p_g-4+\frac{1}{3}$ of type (d).
	\end{example}
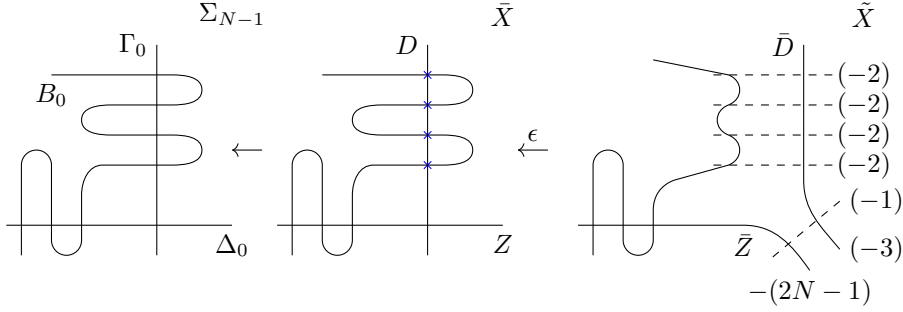
	\begin{tikzpicture}
		\begin{scope}[scale=0.4] 
			\draw 
			(-2,-1) to (-2,2) to[in=180,out=90] (-1.5,2.5) to[in=90,out=0]
			(-1,2) to
			(-1,-0.5) to[in=180,out=270]  (-0.5,-1) to[in=270,out=0]
			(0,-0.5) to (0,1) to[in=180,out=90]  (0.7,2) to 
			(3,2) to[in=270,out=0] (4,2.5) to[in=0,out=90]  
			(3,3) to (1,3) to[in=270,out=180] (0,3.5) to[in=180,out=90] (1,4) to (3,4) to[in=270,out=0] (4,4.5) to[in=0,out=90] 
			(3,5) to (-1,5)node[below]{$B_0$};
			\draw (-2.5,0)  to (5,0) node[below]{$\Delta_0$}; 
			\draw (2.5,-1)  to (2.5,6) node[left]{$\Gamma_0$}; 
			\node at (5,7) {$\Sigma_{N-1}$};
			\draw[->] (6,2.5) to (5,2.5);
		\end{scope}
		\begin{scope}[scale=0.4,xshift=9cm] 
			\draw (-2,-1) to (-2,2) to[in=180,out=90] (-1.5,2.5) to[in=90,out=0]
			(-1,2) to (-1,-0.5) to[in=180,out=270]  (-0.5,-1) to[in=270,out=0]
			(0,-0.5) to (0,1) to[in=180,out=90]  (0.7,2) to 
			(3,2) to[in=270,out=0] (4,2.5) to[in=0,out=90]  
			(3,3) to (1,3) to[in=270,out=180] (0,3.5) to[in=180,out=90] (1,4) to (3,4) to[in=270,out=0] 
			(4,4.5) to[in=0,out=90] 
			(3,5) to (-1,5); 
			\foreach \y in {(2.5,2),(2.5,3),(2.5,4),(2.5,5)}
			\draw[blue] plot[mark=x, mark size=5] coordinates{\y};
			\draw (-2.5,0)  to (5,0) node[below]{$Z$}; 
			\draw (2.5,-1)  to (2.5,6) node[left]{$D$}; 
			\node at (5,7) {$\bar{X}$};
			\draw[->] (6.5,2.5) to  node[above]{$\epsilon$} (5.5,2.5);
		\end{scope}
		\begin{scope}[scale=0.4,xshift=19cm] 
			\draw 
			(-2,-1) to (-2,2) to[in=180,out=90] (-1.5,2.5) to[in=90,out=0]
			(-1,2) to
			(-1,-0.5) to[in=180,out=270]  (-0.5,-1) to[in=270,out=0]
			(0,-0.5) to (0,0.5) to[in=200,out=90]  (0.7,1.5) to 
			(2.5,2)  .. controls (3,2.2) and (3,2.8) .. (2.5,3)
			.. controls (2,3.2) and (2,3.8) .. (2.5,4)
			.. controls (3,4.2) and (3,4.8) .. (2.5,5)
			(2.5,5) to (0,5.5);
			\draw (-2.5,0) to (3,0) node[below]{$\bar{Z}$} to[out=0, in=125] (5.2,-1.5) node[below]{$-(2N-1)$}; 
			\draw (6.2,-0.8) node[right]{$(-3)$} to[in=270, out=130] (5,1.5)  to (5,6) node[left]{$\bar{D}$}; 
			\draw[dashed] (6.2,0.8) node[right]{$(-1)$}  to (3.8,-1.2);
			\foreach \y in {2,3,4,5}
			{	\draw[dashed] (2,\y) to (6,\y);
				\node at (7,\y) {$(-2)$};}
			\node at (7,7) {$\tilde{X}$};
		\end{scope}
	\end{tikzpicture}
	\section{stable surfaces with $|K_X|$ composed with a pencil}
	
	In this section  we assume that $X$ is a normal KSBA stable surface with $|K_X|$ composed with a pencil. We use the notations as in \textsection \ref{ncp}.  We have the following commutative diagram:
	\[
	\xymatrix@C=4em@R=8ex{
		\bar{X}\ar[d]_{\varphi}\ar[drr]_{\phi_{|L|}}\ar[r]^{\epsilon} &\tilde{X}\ar[r]^{\pi}\ar@{-->}[dr]^{\phi_{|K_{\tilde{X}}|}} &X\ar@{-->}[d]^{\phi_{|K_X|}}\\
		B\ar[rr]_{h}&&\mathbb{P}^{p_g-1},
	}\]
	where  $\phi_{|L|}=h_{\circ}\varphi$ is the Stein factorization. Since $|K_X|$ is composed with a pencil, $B$ is a curve.  
	At this time, $L$ is algebraically equivalent to $nF$, where $F$ is the general fiber of $\varphi$ with $F^2=0$ and $n=\mathrm{deg}\, h^*\mathcal{O}_{\mathbb{P}^{p_g-1}}(1)$. 
	We have 
	\begin{align*}
		K_X^2=(\epsilon^*\pi^*K_X)^2=&(nF+G)\epsilon^*\pi^*K_X\ge nF\epsilon^*\pi^*K_X= n(FG-\sum a_E\epsilon^*EF).
	\end{align*}
	
	If $F\epsilon^*\pi^*K_X\ge 2$, then $K_X^2\ge 2n\ge 2(p_g-1)$. 
	If $F\epsilon^*\pi^*K_X=FG-\sum a_E\epsilon^*EF < 2$, then $FG=0$ or $1$ since $-\sum a_E\epsilon^*E\ge 0$. 
	If $FG=0$, then $\epsilon=\mathrm{id}$ by Lemma \ref{firstineq}. If $FG=1$, then $FG\ge FR\ge F(\sum Z_i^e+ \sum Z_j^c+\sum \mathcal{E}_k)$ implies $\epsilon$ is either an identity, a blowup of a smooth point, a resolution of an elliptic Gorenstein singularity, or a resolution of an $ADE$-singularity.
	There are the following possibilities:
	\begin{itemize}
		
		\item[(A)] $FG=1$, $\epsilon$ is a blowup of a smooth point $p$. Hence  $K_{\bar{X}}=nF+2\mathcal{E}+G_0$, where $\mathcal{E}$ is the (-1) curve and $G_0=G-\mathcal{E}\ge0 $. Since $1=FG\ge F\mathcal{E}>0$, we see $F\mathcal{E}=1$, $FG_0=0$ and
		$p_a(F)=1+\frac{FK_{\bar{X}}}{2}=2$. Hence $\widetilde{X}$ admits a genus two fibration over $\mathbb{P}^1$ with a section.  Moreover,  $-1=\mathcal{E}K_{\widetilde{X}}=n\mathcal{E}F-2+\mathcal{E}G_0$ implies $\mathcal{E}G_0=0$,  $n=1$ and $p_g=2$;
		
		\item[(B)] $FG=1$, $\epsilon$ is a resolution of an elliptic Gorenstein singularity. Hence  $K_{\bar{X}}=nF+G_0$, where $G_0=G-Z\ge 0$ and $Z$ is the numerical fundamental cycle associated to the elliptic Gorenstein singularity. Since $1=FG\ge FZ>0$, we have $FZ=1$ and $FG_0=0$. Therefore,
		$p_a(F)=1$ and $Z$ is an elliptic curve. Hence $\widetilde{X}$ admits an elliptic fibration over a simple elliptic curve with a section (i.e. a Jacobian fibration). In this case, $n=p_g$;
		
		\item[(C)] $FG=0$, $\epsilon=\mathrm{id}$. Hence $K_{\widetilde{X}}=nF+G$, and $p_a(F)=1$. Since $F\pi^*K_X=-\sum a_EEF>0$, we have $E_0F>0$ for some $E_0$. Hence $\widetilde{X}$ admits an elliptic fibration over $\mathbb{P}^1$. In this case, $n=p_g-1$;
		
		\item[(D)] $FG=1$, $\epsilon=\mathrm{id}$ or $\epsilon$ is a resolution of ADE singularity. However, in either case, $p_a(F)=1+\frac{1}{2}(K_{\bar{X}}+F)F=1+\frac{1}{2}(nF+G+F)F=1+\frac{1}{2}$, which is impossible.
	\end{itemize}
	
	First we consider case (A). Let $\mu\colon \widehat{X}\to \bar{X}$ be a resolution of $ADE$ singularities on the support of $G_0$.  
	Let $\theta=\pi_{\circ}\epsilon_{\circ}\mu$. Write $\widehat{F}=\mu^*F$, $\widehat{G}_0=\mu^*G_0$, $\widehat{\mathcal{E}}=\mu^*\mathcal{E}$ and $\widehat{E}=\mu^*\epsilon^*E$. 
	$\theta^*K_X=\widehat{F}+\widehat{\mathcal{E}}+\widehat{G}_0-\sum a_{E}\widehat{E} $. 
	
	Since $\widehat{G}_0\widehat{F}=G_0F=0$, $\widehat{G}_0$ is supported on several reducible fibers of $\mu_{\circ} \varphi$. Therefore by the Zariski's lemma, the intersection matrix of the support of $\widehat{G}_0$ is semi-negative definite. Since  $\widehat{G}_0\widehat{\mathcal{E}}=0$ and $\widehat{\mathcal{E}}$ is a 1-section of the fibration, $\widehat{G}_0$ is supported on fiber components disjoint  from $\widehat{\mathcal{E}}$. Thus, the intersection matrix of the support of $\widehat{G}_0$ is negative definite. 
	
	Note that $K_{\widehat{X}}=\widehat{F}+2\widehat{\mathcal{E}}+\widehat{G}_0$ and $(\widehat{F}+2\widehat{\mathcal{E}})\widehat{G}_0=0$. 
	Therefore, by Lemma \ref{blowdown}, there exists $\eta \colon \widehat{X}\to Y$ which is a succession of blow-ups, such that $\eta^*K_Y=\widehat{F}+2\widehat{\mathcal{E}}$ and $\widehat{G}_0=\sum \mathcal{E}_k$, where $\mathcal{E}_k$ is a pullback of a (-1) curve.
	We have the following commutative diagram:
	\[
	\xymatrix{
		\widehat{X}\ar[d]^{\eta} \ar@(ur,ul)[rrr]^{\theta}\ar[r]^{\mu}&\bar{X}\ar[r]^{\epsilon}\ar[d]_{\varphi}&\widetilde{X}\ar[r]^{\pi}& X\\
		Y\ar[r]&B.
	}
	\]

	First we have 
	\begin{align*}
		K_X^2&=(\theta^*K_X)^2=(\widehat{F}+\widehat{\mathcal{E}}+\widehat{G}_0)\theta^*K_X \\
		&=\widehat{F}\theta^*K_X+\widehat{G}_0\theta^*K_X\\
		&= 1-\sum a_{E}\widehat{E}\widehat{F}+\widehat{G}_0\theta^*K_X\\
		&\ge 1.
	\end{align*}
	
	If '=' holds, we have $\sum a_{E}\widehat{E}\widehat{F}=\widehat{G}_0\theta^*K_X=0$.  First we claim that $\widehat{G}_0=0$.  Otherwise, since $\widehat{G}_0$ consists of pullback of (-1)-curves, which is not $\theta$-exceptional, we have $\widehat{G}_0\theta^*K_X>0$, a contradiction. Next we claim that $\pi=\mathrm{id}$. Otherwise, $\sum a_{E}\widehat{E}\widehat{F}=\sum a_{E}\epsilon^*EF=0$ and $F\mathcal{E}=1$ implies that for each $\pi$-exceptional curve $E$ the blowup center $p\not \in E$ and $E$  is contained in a reducible fiber. Therefore $EK_{\widetilde{X}}=\epsilon^*E\epsilon^*K_{\widetilde{X}}=\epsilon^*E(F+\mathcal{E})=0$, which implies each $\pi$-exceptional curve $E$ is a (-2) curve, 
	a contradiction. 
	Therefore, $K_X$ is Cartier.
	By \cite[Thm 3.3]{FPR15a},  $X$ is canonically  embedded as a hypersurface of degree 10 in the smooth locus of $\mathbb{P}(1,1,2,5)$.
	
	If $\pi=\mathrm{id}$ and $\widehat{G}_0\theta^*K_X>0$, then $K_X$ is Cartier and thus $\widehat{G}_0\theta^*K_X=\theta(\widehat{G}_0)K_X\ge1$. Therefore, 
	$$K_X^2= 1+\widehat{G}_0\theta^*K_X=1+\theta(\widehat{G}_0)K_X\ge 2.$$
	
	If $\pi\not=\mathrm{id}$, i.e. $-\sum a_{E}\widehat{E}\not =0$, we claim that $-\sum a_{E}\widehat{E}\widehat{F}\not =0$.
	Otherwise, we assume $-\sum a_{E}\widehat{E}\widehat{F}  =0$ for a contradiction.
	We see that $\widehat{E}\widehat{F}=\epsilon^*EF=0$  for each $E\in \mathrm{supp}(-\sum a_{E}E)$. 
	Hence each $\widehat{E}$ is contained in a reducible fiber, $p\not \in E$ and thus $\widehat{E}\widehat{\mathcal{E}}=\epsilon^*E\mathcal{E}=0$. 
	As $\eta(\widehat{E})K_Y=\widehat{E}\eta^*K_Y=\widehat{E}(\widehat{F}+2\widehat{\mathcal{E}})=0$, $\eta(\widehat{E})$ is either a point or a (-2) curve on $Y$. 
	Since $\widehat{G}_0\widehat{\mathcal{E}}=0$ and $-\sum a_{E}\widehat{E}\widehat{\mathcal{E}}=0$, we see that $\eta_*(-\sum a_{E}\widehat{E})$ is supported on (-2) curves disjoint from the 1-section $\eta(\widehat{\mathcal{E}})$. As a result, $\eta_*(-\sum a_{E}\widehat{E})$ is contractible. Thus  $\widehat{G}_0 -\sum a_{E}\widehat{E}$  is contractible. Therefore, $(\widehat{G}_0 -\sum a_{E}\widehat{E})^2\le 0$. We have
	$$0\le \theta^*K_X(\widehat{G}_0 -\sum a_{E}\widehat{E})=(\widehat{F}+\widehat{\mathcal{E}}+\widehat{G}_0-\sum a_{E}\widehat{E} )(\widehat{G}_0 -\sum a_{E}\widehat{E})=(\widehat{G}_0 -\sum a_{E}\widehat{E})^2\le 0.$$
	Therefore, we have $\widehat{G}_0 =\sum a_{E}\widehat{E}=0$, a contradiction.
	Guaranteed the claim, we may assume that there is a curve $\pi$-exceptional curve $E'$ on $\widetilde{X}$, such that $\widehat{E'}\widehat{F}=\epsilon^*E'F>0$, and hence $E'\epsilon(F)>0$. While $G_0F=0$ implies $\epsilon(G_0)\epsilon(F)=0$, we see that  $\mathrm{Supp}(\epsilon(G_0))\not \supseteq E'$. Therefore $\epsilon(G_0)E'\ge 0$ and thus $\epsilon^*E'G_0\ge 0$. 
	We have $E'K_{\widetilde{X}}=\epsilon^*E'\epsilon^*K_{\widetilde{X}}=\epsilon^*E'(F+\mathcal{E}+G_0)
	=\epsilon^*E'F+\epsilon^*E'G_0\ge \epsilon^*E'F>0$.
	Therefore, we have $-E'^2\ge -3$ and then $-a_{E'}\ge \frac{1}{3}$ by Remark \ref{discrep}. 
	Hence we have
	\begin{align*}
		K_X^2&=1-\sum a_{E}\widehat{E}\widehat{F}+\widehat{G}_0\theta^*K_X\\
		&\ge 1-\sum a_{E}\widehat{E}\widehat{F}\\
		&\ge 1-a_{E'}\\
		&\ge 1+\frac{1}{3}.
	\end{align*}

	Moreover, if '=' holds, then $\widehat{G}_0\theta^*K_X=0$, $\sum\limits_{E\not=E'} a_{E}\widehat{E}\widehat{F}=0$, $\widehat{E'}\widehat{F}=1$ and $a_{E'}=-\frac{1}{3}$.
	Similarly to the the previous case, $\widehat{G}_0\theta^*K_X=0$ implies $G_0=0$, and then $\mu=\mathrm{id}$.
	$a_{E'}=-\frac{1}{3}$ implies $-E'^2=3$ and $E'\sum\limits_{E\not=E'}E=0$.
	We claim that $-\sum a_{E}E=\frac{1}{3}E'$.
	Otherwise, there exists a $\pi$-exceptional curve $E\not=E'$ with $-E^2\ge 3$. 
	However, $\widehat{E}\widehat{F}=0$ implies $\epsilon^*EF=0$. Note that $\epsilon^*EK_{\bar{X}}=\epsilon^*E(F+2\mathcal{E})=0$. 
	Hence $\epsilon^*E$ is a (-2) curve and then $E$ is a (-2) curve, a contradiction. 
	Summarize, if $K_X^2=1+\frac{1}{3}$, $\bar{X}$ admits a genus two fibration with two 1-sections $\mathcal{E}$ and $E'$. $\mathcal{E}^2=-1$, $E'^2=-3$ and $K_{\bar{X}}=F+2\mathcal{E}$. $X$ is obtained from $\bar{X}$ by contracting $\mathcal{E}$ and $E'$. 
	Example \ref{6.1a} shows that such surface exists. Therefore, the inequality is sharp.
	
	For case (B), we define $\mu$, $\theta$ similarly to case A). Write $\widehat{F}=\mu^*F$, $\widehat{G}_0=\mu^*G_0$, $\widehat{Z}=\mu^*Z$ and $\widehat{E}=\mu^*\epsilon^*E= \mu^{-1}\epsilon^{-1}(E)$. 
	We have $\theta^*K_X=n\widehat{F}+\widehat{Z}+\widehat{G}_0-\sum a_{E}\widehat{E} $. 
	Moreover, $K_{\widehat{X}}=n\widehat{F}+\widehat{G}_0$ and $\widehat{F}\widehat{G}_0=0$. 
	Similarly to case (A), by Lemma \ref{blowdown} there exists $\eta \colon \widehat{X}\to Y$ which is a succession of blow-ups, such that  $\eta^*K_Y=n\widehat{F}$ and
	$\widehat{G}_0= \sum \mathcal{E}_k$, where $\mathcal{E}_k$'s are the pullbacks of a (-1) curves.
	We have the following commutative diagram:
	\[
	\xymatrix{
		\widehat{X}\ar[d]^{\eta} \ar@(ur,ul)[rrr]^{\theta}\ar[r]^{\mu}&\bar{X}\ar[r]^{\epsilon}\ar[d]_{\varphi}&\widetilde{X}\ar[r]^{\pi}& X\\
		Y\ar[r]&B.
	}
	\]
	Note that $\varphi$ is a fibration onto an elliptic curve. Therefore, $F\epsilon^*E=0$ for each $\pi$-exceptional curve $E$. Hence $\widehat{F}\widehat{E}=0$ for each $\pi$-exceptional curve $E$.
	We have 
	\begin{align*}
		K_X^2&=(\theta^*K_X)^2=(n\widehat{F}+\widehat{Z}+\widehat{G}_0-\sum a_{E}\widehat{E})\theta^*K_X\\ &=n\widehat{F}\theta^*K_X+\widehat{G}_0\theta^*K_X\\
		&= n+\widehat{G}_0\theta^*K_X\\
		&\ge n\\
		&=p_g.
	\end{align*}
	If '=' holds, we have $\widehat{G}_0\theta^*K_X=0$. 
	Similarly to case A),  $\widehat{G}_0\theta^*K_X=0$ implies $G_0=0$ and $\mu=\mathrm{id}$.
	We claim that $\pi=\mathrm{id}$. Otherwise, there exists a $\pi$-exceptional curve $E$ with $-E^2\ge 3$. 
	However, $EK_{\widetilde{X}}=\epsilon^*EK_{\bar{X}}=\epsilon^*E(nF)=n\widehat{E}\widehat{F}=0$, which implies $E$ is a (-2) curve, a contradiction. 
	Therefore, $K_X$ is Cartier, $\bar{X}$ is a Jacobian surface over an elliptic curve, and $X$ is obtained from $\bar{X}$ by contracting the zero-section $Z$. It is easy to see that every fiber of $\bar{X}$ is irreducible by the ampleness of $K_X$.

	Assume $K_X^2>p_g$, i.e. $\widehat{G}_0\theta^*K_X>0$. 
	Since $\widehat{Z}$ is a 1-section, we see that each fiber is reduced and $C\widehat{Z}\le 1$ for any reduced and irreducible $C$ contained in a fiber. 
	Note that $\widehat{F}\widehat{G}_0=FG_0=0$ implies that $\widehat{G}_0$ is supported on several fibers.  
	We claim that there exists a (-1) curve $\mathcal{E}_0$ in $\mathrm{Supp}(\widehat{G}_0)$ such that $\mathcal{E}_0\widehat{Z}=1$. Otherwise, assume $\widehat{G}_0 \widehat Z=0$ for contradiction. 
	Since $0=E(nF)=\widehat E(n\widehat F)=\widehat E\eta^*(K_Y)=\eta(\widehat E)K_Y $,  $\eta(\widehat E)$ is either a point or a (-2) curve on $Y$. 
	As $\widehat{G}_0 \widehat Z=0$ and $ \sum a_{E}\widehat{E} \widehat Z=0$, $\eta_*(\sum a_{E}\widehat{E})$ is supported on fiber components disjoint from the 1-section $\eta(\widehat Z)$, which are contractible. 
	As a result,  $\widehat{G}_0 -\sum a_{E}\widehat{E}$  is contractible. Therefore, $(\widehat{G}_0 -\sum a_{E}\widehat{E})^2\le 0$. 
	We have
	$$0\le \theta^*K_X(\widehat{G}_0 -\sum a_{E}\widehat{E})=(n\widehat{F}+\widehat{Z}+\widehat{G}_0-\sum a_{E}\widehat{E})(\widehat{G}_0 -\sum a_{E}\widehat{E})=(\widehat{G}_0 -\sum a_{E}\widehat{E})^2\le 0.$$
	Therefore, $\widehat{G}_0 =\sum a_{E}\widehat{E}=0$, a contradiction.
	
	Since $0<\mathcal{E}_0\theta^*K_X=\mathcal{E}_0\widehat{Z}+\mathcal{E}_0\widehat{G}_0-\mathcal{E}_0 \sum a_{E}\widehat E=-\mathcal{E}_0 \sum a_{E}\widehat E$, we see that there exists $\widehat E_1$ such that $\mathcal{E}_0 \widehat E_1 =1$. 
	Locally  around  $\mathcal{E}_0$, the configuration of $\widehat Z$, $\mathcal{E}_0$ and $\widehat E_1$ is as in Figure \ref{localconfg}. 
	
	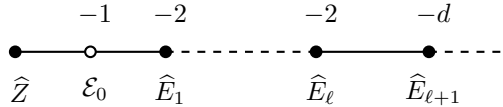
\begin{figure}[H]
		\centering
		\begin{tikzpicture}[thick] 
			\begin{scope}[scale=1.0] 
				\clip(-1,-1) rectangle (7,1);
				\draw (0,0)-- (2,0); 
				\draw[dashed] (2,0)-- (4,0);
				\draw (4,0)-- (5.5,0);
				\draw[dashed] (5.5,0)-- (6.5,0);
				\draw[fill=black] (0,0) circle[radius=2pt];
				\draw[fill=white] (1,0) circle[radius=2pt];
				\foreach \x in {2,4,5.5}
				\draw[fill=black] (\x,0) circle[radius=2pt];
				\foreach \x in {2,4}
				\node at (\x,0.5) {{ $-2$}};
				\node at (1,0.5) {{ $-1$}};
				\node at (5.5,0.5) {{  $-d$}};
				\node at (0,-0.5) {{  $\widehat Z$}};
				\node at (1,-0.5) {{  $\mathcal{E}_0$}};
				\node at (2,-0.5) {{  $\widehat E_1$}};
				\node at (4,-0.5) {{  $\widehat E_{\ell}$}};
				\node at (5.5,-0.5) {{  $\widehat E_{\ell+1}$}};
			\end{scope}
		\end{tikzpicture} 
		\caption{local configurations of $\mathcal{E}_0$}
		\label{localconfg}
	\end{figure}
	
	Note that $d\ge 3$. 
	Let $\mathcal{E}_1=\mathcal{E}_0+\widehat E_1$,..., and $\mathcal{E}_{\ell}=\mathcal{E}_0+\widehat E_1+...+\widehat E_{\ell}$. We see that $\mathcal{E}_{k}^2=-1$, $k=1,...,\ell$ and $\widehat G_0\ge \sum_{k=0}^{\ell}\mathcal{E}_k$. Using the equations $0=\widehat{E}_k\theta^*K_X=\widehat{E}_kK_{\widehat{X}}-\sum a_{E}\widehat{E}\widehat{E}_k$, $k=1,...,l$, we obtain $a_{E_k}=ka_{E_1}$, $k=1,...,\ell+1$.
	
	We have
	\begin{align*}
		K_X^2&= p_g+\widehat G_0\theta^*K_X\\
		&\ge p_g+ \sum_{k=0}^{\ell}\mathcal{E}_k\theta^*K_X\\
		&=p_g+(\ell+1)\mathcal{E}_0\theta^*K_X\\
		&=p_g-(\ell+1)a_{E_1}\\
		&=p_g-a_{E_{\ell+1}}\\
		&\ge p_g+\frac{d-2}{d} \quad\text{(by Remark \ref{discrep})}\\
		&\ge p_g+\frac{1}{3}.
	\end{align*}

	If '=' holds, we have $\widehat G_0\theta^*K_X=\sum_{k=0}^{\ell}\mathcal{E}_k\theta^*K_X$ and $-a_{E_{\ell+1}}=\frac{d-2}{d}=\frac{1}{3}$. 
	$\widehat G_0-\sum_{k=0}^{\ell}\mathcal{E}_k$ consists of sum of pullbacks of (-1) curves. 
	$(\widehat G_0-\sum_{k=0}^{\ell}\mathcal{E}_k)\theta^*K_X=0$ implies $\widehat{G_0}= \sum_{k=0}^{\ell}\mathcal{E}_k$.
	$-a_{E_{\ell+1}}=\frac{1}{3}$ implies $-E_{\ell+1}^2=-3$ and $E_{\ell+1}\sum\limits_{E\not=E_{\ell+1}}E=0$.
	Hence $\ell=0$, and thus $\widehat{G_0}=\mathcal{E}_0$. As a result, we have $\mu=\mathrm{id}$.
	Moreover, we claim that $-\sum a_EE=\frac{1}{3}E_1$. Otherwise, there exists a $\pi$-excepptional curve $E'\not=E_1$ with $-E'^2\ge 3$. 
	Note that $\mathcal{E}_0$, $\widehat{E}_1$ is contained in the same reduced fiber denoted by $\widehat{F}'$ and $\widehat{E}_1$ is the only curve in $\mathrm{Supp}(\widehat{F}'-\mathcal{E}_0)$ satisfying $\widehat{E}_1\cap\mathcal{E}_0\not=\emptyset$. Therefore,  $\mathcal{E}_0\widehat{E}'=0$.
	Thus $E'K_{\tilde{X}}=\widehat{E}'\mu^*\epsilon^*K_{\tilde{X}} =\widehat{E}'\mathcal{E}_0=0$, which implies $E'$ is a (-2) curve, a contradiction. 
	Moreover, $\widehat{G}_0=\mathcal{E}_0$ implies $\eta$ is a blowup of the point $\eta(\widehat{Z})\cap \eta(\widehat{E}_1)$.
	In conclusion, $\bar{X}$ is a blowup of a Jacobian surface $Y$ over an elliptic curve, and $X$ is a contraction of the zero section $Z$ and a (-3) curve on $\bar{X}$.                                                                                                                                                                                                                                                                                                                                                                                                                                                                                                                                                                                                                                                                                                                                                                                                                                                                                                                                                                                                                                                                                                                                                                                                                                                                                                                                                                                                                                                                                                                                                                                                                                                                                                                                                                                                                                                                                                       
	Example \ref{6.1b} shows that such surface exists. Therefore, the inequality is sharp.

	Finally, we consider case (C).
	We define $\mu$, $\theta$ similarly to case A). Write $\widehat{F}=\mu^*F$, $\widehat{G}=\mu^*G$, and $\widehat{E}=\mu^*E= \mu^{-1}(E)$. 
	We have $\theta^*K_X=n\widehat{F}+\widehat{G}-\sum a_{E}\widehat{E} $. 
	Moreover, $K_{\widehat{X}}=n\widehat{F}+\widehat{G}$ and $\widehat{F}\widehat{G}=0$. 
	Similarly to case (A), by Lemma \ref{blowdown} there exists $\eta \colon \widehat{X}\to Y$ which is a succession of blow-ups, such that  $\eta^*K_Y=n\widehat{F}$ and
	$\widehat{G}= \sum \mathcal{E}_k$, where $\mathcal{E}_k$'s are the pullbacks of a (-1) curves.
	We have the following commutative diagram:
	\[
	\xymatrix{
		\widehat{X}\ar[d]^{\eta} \ar@(ur,ul)[rrr]^{\theta}\ar[r]^{\mu}&\bar{X}\ar@{=}[r]\ar[d]_{\varphi}&\widetilde{X}\ar[r]^{\pi}& X\\
		Y\ar[r]&B.
	}
	\]

	On the other hand, 
	since  $G$ is vertical ($FG=0$)  and $E_0$ is horizontal ($E_0F>0$), we have $E_0G\ge 0$. Hence $-E_0^2=2+E_0K_{\widetilde{X}}=2+nE_0F+E_0G\ge 2+nE_0F\ge n+2$. 
	Therefore, we have $-E_0^2\ge n+2$ and then by Remark \ref{discrep} we have
	$$-a_{E_0}=\frac{-2-E_0^2-\sum\limits_{E\not=E_0} a_{E}EE_0}{-E_0^2}\ge \frac{-2-E_0^2}{-E_0^2}\ge \frac{n}{n+2}.$$
	We then have 
	\begin{align}
		K_X^2&=\pi^*K_X^2=(nF+G)\pi^*K_X \nonumber\\
		&\ge nF\pi^*K_X \label{ineq1}\\
		&=-n\sum a_{E}EF \nonumber\\
		&\ge -na_{E_0}E_0F \label{ineq2}\\
		&\ge -na_{E_0} \label{ineq3}\\
		&\ge \frac{n^2}{n+2} \label{ineq4}\\
		&=\frac{p_g-1}{p_g+1}(p_g-1). \nonumber
	\end{align}
	
	If $K_X^2=\frac{p_g-1}{p_g+1}(p_g-1)$ i.e. (\ref{ineq1}), (\ref{ineq2}), (\ref{ineq3})  and (\ref{ineq4}) achieve '=', we have $G\pi^*K_X=0$, $\sum\limits_{E\not=E_0} a_{E}EF=0$, $E_0F=1$, $E_0^2=-n-2$, and $\sum\limits_{E\not=E_0}EE_0=0$. 
	We claim $G=0$. Otherwise,  since $\widehat{G}$ consists of pullback of (-1)-curves, which is not $\theta$-exceptional, we have $\widehat{G}\theta^*K_X=G\pi^*K_X>0$, a contradiction. 
	As a result, $\mu=\mathrm{id}$. 
	Next we claim $\sum a_{E}E=-\frac{n}{n+2}E_0$. 
	Otherwise, there exists a $\pi$-exceptional curve $E'\not =E_0$ with $-E'^2\ge 3$. Moreover, $\sum\limits_{E\not=E_0} a_{E}EF=0$ implies $E'F=0$. Thereofore, we have $E'K_{\tilde{X}}=nE'F=0$, a contradiction.
	Hence $\widetilde{X}$ is a Jacobian surface over $\mathbb{P}^1$ with a zero section $E_0$ and $E_0^2=-n-2=-p_g-1$.
	$X$ is obtained from $\tilde{X}$ by contracting the zero section $E_0$. 
	Moreover, it is easy to see that every fiber of $\widetilde{X}$ is irreducible by the ampleness of $K_X$. 
	
	If $K_X^2>\frac{p_g-1}{p_g+1}(p_g-1)$, we may assume $\frac{p_g-1}{p_g+1}(p_g-1)<K_X^2<\frac{p_g-1}{p_g+1}(p_g-1)+1$. We claim that at least one of the inequalities  (\ref{ineq2}), (\ref{ineq3})  or (\ref{ineq4})  is strict. Otherwise, we assume $F\pi^*K_X=\frac{p_g-1}{p_g+1}$ for a contradiction.  We have $F\pi^*K_X=-\sum\limits_{E\not=E_0} a_{E}EF=\frac{n}{n+2}=\frac{p_g-1}{p_g+1}$. Hence $E_0F=1$ and $E_0G=E_0\sum\limits_{E\not=E_0}E=\sum\limits_{E\not=E_0}EF=0$. 
	Moreover, we have $G>0$ since  (\ref{ineq1}) is strict. 
	If $\sum\limits_{E\not=E_0}a_EE=0$, we have $0<\widehat{G}\theta^*K_X=\widehat{G}K_{\widehat{X}}-a_{E_0}\widehat{E}_0\widehat{G}=\widehat{G}K_{\widehat{X}}<0$, a contradiction. 
	If $\sum\limits_{E\not=E_0}a_EE \not=0$, then for each $\pi$-exceptional curve $E\not=E_0$, $\eta(\widehat{E})K_Y=n\widehat{E}\widehat{F}=0$, which implies $\eta(\widehat{E})$ is either a point or a (-2) curve. Moreover, $0=E_0G=\widehat{E}_0\widehat{G}$ implies that the blowup centers of $\eta$ is outside $\eta(\widehat{E}_0)$. Therefore, $\eta(\sum\limits_{E\not=E_0}\widehat{E})$ is still disjoint with  $\eta(\widehat{E}_0)$ which is a 1-section of the fibration. 
	Therefore, $\eta_*(\sum\limits_{E\not=E_0}a_E\widehat{E} )$ is contractible. As a result, $\widehat{G}-\sum\limits_{E\not=E_0}a_E\widehat{E} $ is contractible. Thus $(\widehat{G}-\sum\limits_{E\not=E_0}a_E\widehat{E} )^2\le 0$. 
	We have
	$$0\le \theta^*K_X(\widehat{G} -\sum\limits_{E\not=E_0}a_E\widehat{E} )=(n\widehat{F}+\widehat{G}-\sum a_{E}\widehat{E})(\widehat{G} -\sum\limits_{E\not=E_0}a_E\widehat{E} )=(\widehat{G} -\sum\limits_{E\not=E_0}a_E\widehat{E} )^2\le 0.$$
	Therefore, $\widehat{G} =\sum\limits_{E\not=E_0}a_E\widehat{E} =0$, a contradiction. 
	Thus, we have proved the claim.

	
	We note that if (\ref{ineq4}) is strict i.e. $-a_{E_0}>\frac{n}{n+2}$, we have either $\sum\limits_{E\not=E_0} EE_0>0$, or $\sum\limits_{E\not=E_0} EE_0=0$ and $-E_0^2>n+2$. 
	
	We distinguish between the following cases: 
	\begin{itemize}
		\item[a)] (\ref{ineq2}) is strict, i.e. there exist $E'\not=E_0$ with $E'F\ge 1$.
		Similarly to $E_0$, we have $-a_{E'}\ge \frac{n}{n+2}$. Thus we have 
		\begin{align*}
			K_X^2&=(nF+G)\pi^*K_X=-n\sum a_{E}EF+G\pi^*K_X\ge-n\sum a_{E}EF\\
			&\ge 2n\cdot\frac{n}{n+2}=\frac{2p_g-2}{p_g+1}(p_g-1). 
		\end{align*}
		Noticing that $\frac{2p_g-2}{p_g+1}(p_g-1)>\frac{2p_g-2}{2p_g+1}(p_g-1)$ which is achieved in case c), we see $K_X^2$ does not achieve the minimal value in this case.

		\item[b)] (\ref{ineq3}) is strict, i.e. $E_0F\ge 2$. We have $-E_0^2\ge 2+2n$. Therefore,
		\begin{align*}
			K_X^2&=(nF+G)\pi^*K_X=-n\sum a_{E}EF+G\pi^*K_X\ge-n\sum a_{E}EF\\
			&\ge 2n\cdot\frac{2n}{2n+2}=\frac{2p_g-2}{p_g}(p_g-1). 
		\end{align*}
		Since $\frac{2p_g-2}{p_g}(p_g-1) > \frac{2p_g-2}{2p_g+1}(p_g-1)$, $K_X^2$ does not achieve the minimal value in this case as well.
		
		\item[c)]  (\ref{ineq2}) and (\ref{ineq3}) achieve '=', and $\sum\limits_{E\not=E_0} EE_0>0$. 
		We have $-\sum\limits_{E\not=E_0} a_{E}EF=0$, $E_0F=1$ and $-E_0^2\ge n+2$.
		By \cite[ Lemma 3.7 (ii,iii)]{alexeev89}, the minimal possible value of $-a_{E_0}$ with $-E_0^2\ge n+2$ and $\sum\limits_{E\not=E_0} EE_0>0$ is $\frac{2n}{2n+3}$ and is achieved when the exception curve consists of $E_0+E_1$, where  $E_0$ is a $-(n+2)$ curve, $E_1$ is a (-2) curve and $E_0E_1=1$. 
		Therefore, 
		\begin{align*}
			K_X^2&=(nF+G)\pi^*K_X=-n\sum a_{E}EF+G\pi^*K_X\\
			&\ge -na_{E_0}E_0F\\
			&\ge n\cdot\frac{2n}{2n+3}=\frac{2p_g-2}{2p_g+1}(p_g-1).
		\end{align*}
		Example \ref{6.1c1} shows that this minimal value can be achieved.

		\item[d)] (\ref{ineq2}) and (\ref{ineq3}) achieve '=', $\sum\limits_{E\not=E_0} EE_0=0$ and $-E_0^2>n+2$.
		At this time, we have $G>0$ since $E_0G=-E_0^2-(2+nE_0F)=-E_0^2-n-2>0$. 
		On the other hand, $0<E_0G=\widehat{E}_0\widehat{G}$ implies $\eta$ has blowup centers on $\eta(\widehat{E}_0)$. 
		Therefore, the local configuration around $E_0$ is as Figure \ref{localconfg},
		where $k\ge 1$, $\mathcal{E}_i$ is (-1) curve and $d_{i,\ell_{i}}\ge 3$, for $i=1,2,...,k$. 
		We remark that $G\ge \sum \limits_{i=1}^k \ell_{i}\mathcal{E}_i$. 

		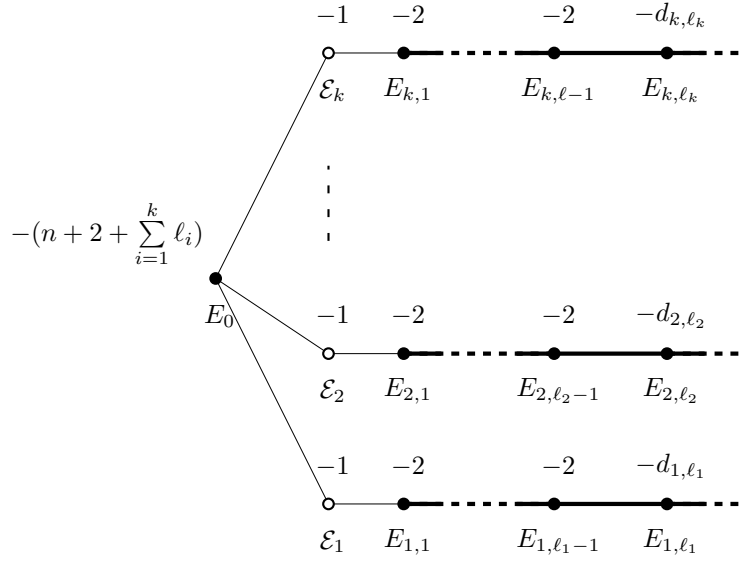
\begin{figure}
			\centering
			\begin{tikzpicture}[thick] 
				\begin{scope}[scale=1.0] 
					\clip(-4,-1) rectangle (7,7);
					
					\node at (-2,3.6) {{ $-(n+2+\sum\limits_{i=1}^k \ell_{i})$}};
					\node at (-0.5,2.5) {{  $E_0$}};
					\draw[fill=black] (-0.5,3) circle[radius=2pt];
					
					\foreach \y in {0,1,3}
					{
						\draw[very thin] (-0.5,3) -- (1,0+\y *2);
					}
					
					\draw[ultra thin] (1,0+6)-- (2,0+6); 
					\draw[ultra thick] (2,0+6)-- (2.5,0+6); 
					\draw[ultra thick, dashed] (2,0+6)-- (4,0+6);
					\draw[ultra thick] (3.5,0+6)-- (6,0+6);
					\draw[ultra thick, dashed] (5.5,0+6)-- (6.5,0+6);
					
					\draw[fill=white] (1,0+6) circle[radius=2pt];
					\foreach \x in {2,4,5.5}
					\draw[fill=black] (\x,0+6) circle[radius=2pt];
					\foreach \x in {2,4}
					\node at (\x,0.5+6) {{ $-2$}};
					\node at (1,0.5+6) {{ $-1$}};
					\node at (5.5,0.5+6) {{  $-d_{k,\ell_k}$}};
					
					\node at (1,-0.5+6) {{  $\mathcal{E}_{k}$}};
					\node at (2,-0.5+6) {{  $ E_{k,1}$}};
					\node at (4,-0.5+6) {{  $ E_{k,\ell-1}$}};
					\node at (5.5,-0.5+6) {{  $ E_{k,\ell_k}$}};
					
					\foreach \y in {1,2}
					{	\draw[ultra thin]  (1,0+\y *2-2)-- (2,+\y *2-2); 
						\draw[ultra thick]  (2,0+\y *2-2)-- (2.5,+\y *2-2); 
						\draw[ultra thick, dashed] (2,0+\y *2-2)-- (4,0+\y *2-2);
						\draw[ultra thick] (3.5,0+\y *2-2)-- (6,0+\y *2-2);
						\draw[ultra thick, dashed] (5.5,0+\y *2-2)-- (6.5,0+\y *2-2);
						
						\draw[fill=white] (1,0+\y *2-2) circle[radius=2pt];
						\foreach \x in {2,4,5.5}
						\draw[fill=black] (\x,0+\y *2-2) circle[radius=2pt];
						\foreach \x in {2,4}
						\node at (\x,0.5+\y *2-2) {{ $-2$}};
						\node at (1,0.5+\y *2-2) {{ $-1$}};
						\node at (5.5,0.5+\y *2-2) {{  $-d_{\y,\ell_{\y}}$}};
						
						\node at (1,-0.5+\y *2-2) {{  $\mathcal{E}_{\y}$}};
						\node at (2,-0.5+\y *2-2) {{  $ E_{\y,1}$}};
						\node at (4,-0.5+\y *2-2) {{  $ E_{\y,\ell_{\y}-1}$}};
						\node at (5.5,-0.5+\y *2-2) {{  $ E_{\y,\ell_{\y}}$}};
					}
					\draw[loosely dashed] (1,0+3.5)-- (1,0+4.5);
				\end{scope}
			\end{tikzpicture} 
			\caption{local configuration around $E_0$}
			\label{localconfg}
		\end{figure}
		We note that $a_{E_{i,j}}=ja_{E_{i,1}}$. Moreover, by \cite[ Lemma 3.7 (ii,iii)]{alexeev89}, we have
		\begin{equation}\label{discrofend}
			-a_{E_{i,1}}\ge \frac{1}{2\ell_i+1},
		\end{equation} 
		and '=' holds if and only if $d_{i,\ell_i}=3$ and $E_{i,\ell_i}$ is an end curve of the exceptional locus, i.e. there is no exceptional curve intersecting  $E_{i,\ell_i}$ except $E_{i,\ell_i-1}$.
		
		When $n\le 3$, we have
		\begin{align*}
			K_X^2&=(nF+G)\pi^*K_X=-n\sum a_{E}EF+G\pi^*K_X\\
			&> -na_{E_0}E_0F\\
			&\ge n\cdot\frac{n+1}{n+3}\\
			&=\frac{p_g}{p_g+2}(p_g-1).
		\end{align*}
		Since $\frac{p_g}{p_g+2}(p_g-1)\ge \frac{2p_g-2}{2p_g+1}(p_g-1)$ which is achieved in case c), $K_X^2$ does not achieve the minimal value in this case.

		Next we deal with the case $n\ge 4$. We will show that $K_X^2$  achieves the minimal value in this case.
		The volume
		\begin{align*}  
			K_X^2&=(nF+G)\pi^*K_X= -n\sum a_{E}EF+G\pi^*K_X\\
			&= -na_{E_0}E_0F+G\pi^*K_X\\
			&\ge -na_{E_0}+\ell_1 \mathcal{E}_1\pi^*K_X\\
			&= -na_{E_0}+\ell_1(-a_{E_0}-a_{E_{1,1}}-1)\\
			&\ge n\frac{n+\ell_1}{n+2+\ell_1}+\ell_1(\frac{n+\ell_1}{n+2+\ell_1}+\frac{1}{2\ell_1+1}-1) \quad \text{(by (\ref{discrofend}))}\\
			&=n - \frac{3}{2} + \frac{15 \ell_1 - n + 6}{\left(4 \ell_1 + 2\right) \left(\ell_1 + n + 2\right)}.
		\end{align*}
		
		Denote $V_{n,\ell}\colon=n - \frac{3}{2} + \frac{15 \ell - n + 6}{\left(4 \ell + 2\right) \left(\ell + n + 2\right)}$. 
		We consider the difference 
		$$\Delta V_{n,\ell}\colon =V_{n,\ell}-V_{n,\ell+1}=\frac{15 \ell ^{2} + \ell  \left(27 - 2 n\right) - n^{2} - 5 n + 6}{\left(2 \ell  + 1\right) \left(2 \ell  + 3\right) \left(\ell + n + 2\right) \left(\ell + n + 3\right)}.$$ 
		Note that $15 \ell ^{2} + \ell  \left(27 - 2 n\right) - n^{2} - 5 n + 6$ increases monotonically with $\ell$ when $\ell$ ranges from $1$ to $\infty$. 
		Moreover, $\Delta V_{n,1}=- \frac{n^{2} + 7 n - 48}{15 \left(n + 3\right) \left(n + 4\right)}<0$ when $n\ge 5$, and $\Delta V_{n,\ell}>0$ when $\ell$ is large enough.  
		$\underset{\ell\to +\infty}{\mathrm{lim}} V_{m,\ell}=n-\frac{3}{2}$.  
		Hence, $V_{n,\ell}$ increases and then decreases to the limit $V_{n,\infty}=n-\frac{3}{2}$ as $\ell$ increases from 1 to $+\infty$. 
		Note that $n-\frac{3}{2}\ge V_{n,1}$, when $n\ge 21$.  
		Therefore, $\mathrm{min}\{V_{n,\ell}|\ell=1,2...\}=V_{n,1}$, when  $n\ge 21$. 
		Thus, we conclude that $K_X^2\ge V_{n,\ell_1}\ge V_{n,1}=\frac{(3p_g-2)p_g-4}{3(p_g+2)}$ when $n=p_g-1\ge 21$. 
		Example \ref{6.1c2} shows that '=' can be achieved.  
		Note that $\frac{(3p_g-2)p_g-4}{3(p_g+2)}<\frac{2p_g-2}{2p_g+1}(p_g-1)$. Thus $\frac{(3p_g-2)p_g-4}{3(p_g+2)}$ is exactly the minimal volume of $K_X^2$.

		It remains to consider the case $4\le n\le 20$. We will show that the minimal volume is $V_{n,1}$ as well.
		We first consider normal KSBA stable surfaces whose local configuration around $E_0$
		is as Figure \ref{localconfg3}, where $k=1$, $G=\ell \mathcal{E}$, $d_{\ell}=3$ and $E_{\ell}$ is an end curve of the exceptional curve.

		\begin{figure}
			\centering
			\begin{tikzpicture}[thick] 
				\begin{scope}[scale=1.0] 
					\clip(-2,-1) rectangle (7,1);
					\draw (-0.5,0)-- (2,0); 
					\draw[very thick] (2,0) -- (2.5,0); 
					\draw[dashed] (2,0)-- (4,0);
					\draw[very thick] (3.5,0)-- (5.5,0);
					\draw[fill=black] (-0.5,0) circle[radius=2pt];
					\draw[fill=white] (1,0) circle[radius=2pt];
					\foreach \x in {2,4,5.5}
					\draw[fill=black] (\x,0) circle[radius=2pt];
					\foreach \x in {2,4}
					\node at (\x,0.5) {{ $-2$}};
					\node at (1,0.5) {{ $-1$}};
					\node at (-0.5,0.5) {{ $-(n+2+\ell)$}};
					\node at (5.5,0.5) {{  $-3$}};
					\node at (-0.5,-0.5) {{  $E_0$}};
					\node at (1,-0.5) {{  $\mathcal{E}$}};
					\node at (2,-0.5) {{  $ E_1$}};
					\node at (4,-0.5) {{  $ E_{\ell-1}$}};
					\node at (5.5,-0.5) {{  $ E_{\ell}$}};
				\end{scope}
			\end{tikzpicture} 
			\caption{local configuration around $\mathcal{E}_0$ with $-a_{E_{\ell}}= \frac{2\ell}{2\ell+1}$}
			\label{localconfg3}
		\end{figure}
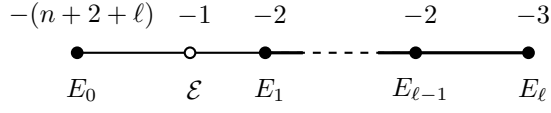

		The volume
		\begin{align*}  
			K_X^2&=(nF+G)\pi^*K_X= -n\sum a_{E}EF+G\pi^*K_X\\
			&= -na_{E_0}E_0F+G\pi^*K_X\\
			&= -na_{E_0}+\ell \mathcal{E}\pi^*K_X\\
			&= -na_{E_0}+\ell(-a_{E_0}-a_{E_1}-1)\\
			&= n\frac{n+l}{n+2+\ell}+\ell(\frac{n+\ell}{n+2+\ell}+\frac{1}{2\ell+1}-1) \\
			&=n - \frac{3}{2} + \frac{15 \ell - n + 6}{\left(4 \ell + 2\right) \left(\ell + n + 2\right)}\\
			&=V_{n,\ell}.
		\end{align*}

		Moreover, since $\mathcal{E}\pi^*K_X=\frac{n+\ell}{n+2+\ell}+\frac{1}{2\ell+1}-1>0$, we have $\ell<\frac{n}{3}$.
		We note that $\mathrm{min} \{ V_{n,\ell}|l\le \frac{n}{3}\}=V_{n,1}$ for $4\le n\le 20$ by direct computation. 
		Therefore, the minimal volume $K_X^2$ is $V_{n,1}$ if the local configuration around $E_0$
		is as Figure \ref{localconfg3}.
		
		For general cases, we discuss according to the value of $k$. 
			
		For the case $k\ge 3$, we have 
		\begin{align*}  
			K_X^2&=(nF+G)\pi^*K_X= -n\sum a_{E}EF+G\pi^*K_X\\
			&= -na_{E_0}E_0F+G\pi^*K_X\\
			&\ge -na_{E_0}+\sum_{i=1}^k \ell_i \mathcal{E}_i\pi^*K_X\\
			&= -na_{E_0}+\sum_{i=1}^k\ell_i(-a_{E_0}-a_{E_{\ell_i,1}}-1)\\
			&= -na_{E_0}+\sum_{i=1}^k\ell_i(-a_{E_0}-\frac{a_{E_{i,\ell_i}}}{\ell_i}-1)\\
			&=n\frac{n+\sum_{i=1}^k\ell_i}{n+2+\sum_{i=1}^k\ell_i}-\frac{2\sum_{i=1}^k\ell_i}{n+2+\sum_{i=1}^k\ell_i}-\sum_{i=1}^k a_{E_{i,\ell_i}} \\
			&=n-2+\frac{4}{n+2+\sum_{i=1}^k\ell_i} -\sum_{i=1}^k a_{E_{i,\ell_i}} \\
			&\ge n-2+\frac{4}{n+2+3} +3\times \frac{1}{3} \\
			&> V_{n,1}.
		\end{align*}
		Therefore, the volume $K_X^2$ does not achieve the minimal value in this case.
		
		For the case $k=2$ and $5\le n \le 20$,  we have 
		\begin{align*}  
			K_X^2 
			&\ge n\frac{n+\sum_{i=1}^2\ell_i}{n+2+\sum_{i=1}^2\ell_i}-\frac{2\sum_{i=1}^2\ell_i}{n+2+\sum_{i=1}^2\ell_i}-\sum_{i=1}^2 a_{E_{i,\ell_i}} \\
			&= n-2+\frac{4}{n+2+\ell_1+\ell_2}+\frac{\ell_1}{2\ell_1+1}-\sum_{i=1}^2 \ell_i a_{E_{i,1}}\\
			&\ge n-2+\frac{4}{n+2+\ell_1+\ell_2}+\frac{\ell_1}{2\ell_1+1}+\frac{\ell_2}{2\ell_2+1}\quad \text{(by (\ref{discrofend}))}.
		\end{align*}
		Since 
		\begin{align*}  
			K_X^2-V_{n,1}&\ge n-2+\frac{4}{n+2+\ell_1+\ell_2}+\frac{\ell_1}{2\ell_1+1}+\frac{\ell_2}{2\ell_2+1}-(n  - \frac{3}{2} + \frac{21 - n}{6 \left(n + 3\right)})\\
			&= \frac{4}{n+2+\ell_1+\ell_2}+\frac{\ell_1}{2\ell_1+1}+\frac{\ell_2}{2\ell_2+1}-\frac{4}{n+3}-\frac{1}{3}\\
			&=\frac{\ell_1}{2\ell_1+1}+\frac{\ell_2}{2\ell_2+1}-\frac{1}{3}-\frac{4 \left(\ell_1+ \ell_2 - 1\right)}{\left(n + 3\right) \left( n +\ell_1+ \ell_2 + 2\right)}\\
			&\ge \frac{\ell_1}{2\ell_1+1}+\frac{\ell_2}{2\ell_2+1}-\frac{1}{3}-\frac{4 \left(\ell_1+ \ell_2 - 1\right)}{8 \left( \ell_1+ \ell_2 + 7\right)}  \quad\	\quad (n\ge 5)\\
			&=\frac{ \ell_2^{2} \left(4 \ell_1 - 4\right) + \ell_2 \left(4 \ell_1^{2} + 116 \ell_1 + 15\right) - 4 \ell_1^{2} + 15 \ell_1 - 11}{6 \left(2 \ell_1 + 1\right) \left(2 \ell_2 + 1\right) \left(\ell_1 + \ell_2 + 7\right)}\\
			&\ge \frac{\left(4 \ell_1 - 4\right) + \left(4 \ell_1^{2} + 116 \ell_1 + 15\right) - 4 \ell_1^{2} + 15 \ell_1 - 11}{6 \left(2 \ell_1 + 1\right) \left(2 \ell_2 + 1\right) \left(\ell_1 + \ell_2 + 7\right)}   \quad\quad (\text{by}~\ell_2\ge 1)\\
			&=\frac{135\ell_1 }{6 \left(2 \ell_1 + 1\right) \left(2 \ell_2 + 1\right) \left(\ell_1 + \ell_2 + 7\right)}\\
			&>0,
		\end{align*}
		we have $K_X^2>V_{n,1}$. Therefore, the volume $K_X^2$ does not achieve the minimal value in this case.
		
		For the case $k=2$ and $n=4$, since $\mathcal{E}_i\pi^*K_X=-a_{E_0}-\frac{a_{E_{i,\ell_i}}}{\ell_i}-1=\frac{4+\sum_{i=1}^2\ell_i}{4+2+\sum_{i=1}^2\ell_i}-\frac{a_{E_{i,\ell_i}}}{\ell_i}-1>0$ and $0<-a_{E_{i,\ell_i}}\le 1$, we have $\ell_1<\ell_2+6$ and $\ell_2< \ell_1+6$. We may assume $\ell_2\le \ell_1<\ell_2+6$. 
		\begin{align*}  
			K_X^2-V_{4,1}&=K_X^2-(4  - \frac{3}{2} + \frac{21 - 4}{6 \left(4 + 3\right)})\\
			&\ge \frac{4}{4+2+\ell_1+\ell_2}+\frac{\ell_1}{2\ell_1+1}+\frac{\ell_2}{2\ell_2+1}-\frac{4}{4+3}-\frac{1}{3}\\
			&=\frac{ \ell_2^{2} \left(8\ell_1 - 17\right) + \ell_2 \left(8\ell_1^{2} + 350\ell_1 + 47\right)- 17\ell_1^{2} + 47\ell_1  - 30}{21 \left(2\ell_1 + 1\right) \left(2 \ell_2 + 1\right) \left(\ell_1+ \ell_2 + 6\right)}.
		\end{align*}
		We see  that $ \ell_2^{2} \left(8\ell_1 - 17\right) + \ell_2 \left(8\ell_1^{2} + 350\ell_1 + 47\right)- 17\ell_1^{2} + 47\ell_1  - 30>0$ by direct computation replacing $\ell_1$ with $\ell_2$,  $\ell_2+1$,..., and  $\ell_2+5$. Thus, $K_X^2>V_{4,1}$. Therefore, the volume $K_X^2$ does not achieve the minimal value in this case.
		
		
		For the case $k=1$ and the local configuration is not as Figure \ref{localconfg3}, by \cite[ Lemma 3.7 (ii,iii)]{alexeev89}
		we have $-a_{E_{\ell}}\ge \frac{2\ell}{3\ell+2}$, where '=' is achieved if the configuration is as Figure \ref{localconfg4}.
		Since $\mathcal{E}\pi^*K_X=\frac{n+l}{n+2+\ell}-a_{E_1}-1=\frac{n+l}{n+2+\ell}-\frac{a_{E_{\ell}}}{\ell}-1>0$, we have $\ell<\frac{(n+2)|a_{E_{\ell}}|}{2-|a_{E_{\ell}}|}<n+2$.
		

		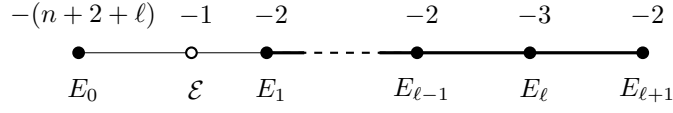
\begin{figure}
			\centering
			\begin{tikzpicture}[thick] 
				\begin{scope}[scale=1.0] 
					\clip(-2,-1) rectangle (8,1);
					\draw[very thin] (-0.5,0) -- (2,0);
					\draw[very thick] (2,0) -- (2.5,0); 
					\draw[dashed,thick] (2,0) -- (4,0);
					\draw[very thick] (3.5,0) -- (7,0);
					\draw[fill=black] (-0.5,0) circle[radius=2pt];
					\draw[fill=white] (1,0) circle[radius=2pt];
					\foreach \x in {2,4,5.5,7}
					\draw[fill=black] (\x,0) circle[radius=2pt];
					\foreach \x in {2,4}
					\node at (\x,0.5) {{ $-2$}};
					\node at (1,0.5) {{ $-1$}};
					\node at (-0.5,0.5) {{ $-(n+2+\ell)$}};
					\node at (5.5,0.5) {{  $-3$}};
					\node at (7,0.5) {{  $-2$}};
					\node at (-0.5,-0.5) {{  $E_0$}};
					\node at (1,-0.5) {{  $\mathcal{E}$}};
					\node at (2,-0.5) {{  $ E_1$}};
					\node at (4,-0.5) {{  $ E_{\ell-1}$}};
					\node at (5.5,-0.5) {{  $ E_{\ell}$}};
					\node at (7,-0.5) {{  $ E_{\ell+1}$}};
				\end{scope}
			\end{tikzpicture} 
			\caption{local configuration around $\mathcal{E}_0$ with $-a_{E_{\ell}}= \frac{2\ell}{3\ell+2}$} 
			\label{localconfg4}
		\end{figure}
		
		Therefore, the volume
		\begin{align*}  
			K_X^2&=n\frac{n+l}{n+2+\ell}+\ell(\frac{n+\ell}{n+2+\ell}-a_{E_1}-1)\\
			&\ge n\frac{n+l}{n+2+\ell}+\ell(\frac{n+\ell}{n+2+\ell}+\frac{2}{3\ell+2}-1) \\
			&=n - \frac{4}{3} + \frac{4 \left(8 \ell - n + 4\right)}{3 \left(3 \ell + 2\right) \left(n + \ell+ 2\right)}.
		\end{align*}
		Denote $W_{n,\ell}\colon=n - \frac{4}{3} + \frac{4 \left(8 \ell - n + 4\right)}{3 \left(3 \ell + 2\right) \left(n + \ell+ 2\right)}$. 
		The sign of the difference $$\Delta W_{n,\ell}\colon =W_{n,\ell}-W_{n,\ell+1}= \frac{4 \left(8 \ell^{2} + \ell \left(16 - 2 n\right) - n^{2} - 5 n + 4\right)}{\left(3 \ell + 2\right) \left(3 \ell + 5\right) \left( n +\ell+ 2\right) \left(n +\ell + 3\right)}$$
		changes at most one time as $\ell$ increases from $\ell=1$ to $+\infty$.
		Therefore, $\mathrm{min}\{W_{n,\ell}|\ell <n+2\}=\mathrm{min}\{W_{n,1},W_{n,n+1}\}$. 
		Note that when $n\ge 4$, we have 
		\begin{align*}
			W_{n,n+1}-V_{n,1}&=\frac{6 n^{3} - 7 n^{2} - 24 n + 9}{3 \left(n + 3\right) \left(2 n + 3\right) \left(3 n + 5\right)}\\
			&=\frac{ (2n- 7) n^{2} +(4n^2- 24) n + 9}{3 \left(n + 3\right) \left(2 n + 3\right) \left(3 n + 5\right)}>0,
		\end{align*}
		and $W_{n,1}>V_{n,1}$. Therefore, we have $\mathrm{min}\{W_{n,\ell}|\ell <n+2\}> V_{n,1}$. The volume $K_X^2$ does not achieve the minimal value in this case.

	\end{itemize}
	
	In conclusion, when $K_X^2>\frac{p_g-1}{p_g+1}(p_g-1)$, 
	we have $K_X^2 \ge \frac{2p_g-2}{2p_g+1}(p_g-1)$ if $p_g\le 4$ and $K_X^2 \ge \frac{(3p_g-2)p_g-4}{3(p_g+2)}$ if $p_g\ge 5$. 
	Moreover, the inequalities are sharp. 
	
	We summarize the above results in the following theorem:
	
	\begin{theorem}\label{main}
		Let $X$ be a normal KSBA stable surface with $|K_X|$ composed with a pencil. Then either $K_X^2\ge 2p_g-2$, or $X$ falls in  one of the following case:
		\begin{itemize}
			
			\item[(A)]  $X$ is birational to a surface $\widetilde{X}$ which admits a genus 2 fibration over $\mathbb{P}^1$ with a section. Moreover, $p_g=2$ and $K_X^2\ge 1$. If '=' holds, $K_X$ is Cartier and $X$ is canonically  embedded as a hypersurface of degree 10 in the smooth locus of $\mathbb{P}(1,1,2,5)$. If '$>$' holds, $K_X^2\ge 1+\frac{1}{3}$ and the inequality is sharp; 
			\item[(B)] $X$ is birational to a Jacobian surface over an elliptic curve. Moreover,  $K_X^2\ge p_g$. If '=' holds, $K_X$ is Cartier and 
			$X$ is obtained from a Jacobian surface $\widetilde X$ (possible singular) over an elliptic curve with every fiber irreducible  by contracting the zero sections.
			If '$>$' holds, $K_X^2\ge p_g+\frac{1}{3}$ and the inequality is sharp;
			\item[(C)] $X$ is birational to an elliptic surface over $\mathbb{P}^1$. Moreover, $K_X^2\ge \frac{p_g-1}{p_g+1}(p_g-1)$. If '=' holds, 
			$X$ is obtained from a Jacobian surface  (possible singular) over $\mathbb{P}^1$ 
			by contracting the zero section and curves in fibers disjoint from the zero section.
			If '$>$' holds, 
			$K_X^2 \ge \mathrm{min}\{\frac{2p_g-2}{2p_g+1}(p_g-1), \frac{(3p_g-2)p_g-4}{3(p_g+2)}\}$, and the inequality is sharp.
			
		\end{itemize}
	\end{theorem}
	
	\begin{example}\label{6.1a}
		Let $\Sigma_2$ be a Hirzebruch surface, $\Delta_0$ be the  zero section with $\Delta_0^2=-2$ and $\Gamma$ be a fiber of the ruling.  We choose a branched curve $B=\Delta_0+B_1+B_2$, where $B_1\sim \Delta_0+2\Gamma$ and $B_2\sim 4(\Delta_0+2\Gamma)$ are two reduced irreducible smooth curves. We may assume $B_1$ intersects $B_2$ with 8 different points. Let $Y$ be the surface which is a double covering surface over $\Sigma_2$ with branched curve $B$. $Y$ has 8 $A_1$ singularities lying over $B_1\cap B_2$. Let $\widetilde{Y}$ be the smooth model of $Y$ and $E_0,E_1$ be the strict transformations of $\Delta_0, B_1$. We see that $E_0^2=-1$ and $E_1^2=-3$. 
		Contracting $E_0$, $E_1$, we obtain a KSBA stable surface $X$ with $K_X^2=1+\frac{1}{3}$ and $p_g=2$, which is an example of Theorem \ref{main} (A). 
		
	\end{example}

	\begin{example}\label{6.1b}
		Let $Y$ be a smooth minimal Jacobian surface over an elliptic curve with $Z$ a zero section and $K_Y\equiv nF$. Assume further $Y$ has  reducible fibers $F_0=\sum_{i=0}^{m_0} C_i$ and $F_{j}=\sum_{i=0}^{m_j} C_i^{(j)}, j=1,...,\ell$. We may assume $C_0Z= C_0^{j}Z=1$, and  $C_iZ= C_i^{(j)}Z=0$, for $i=1,...,m_j$ and $j=1,...,\ell$. Let $\epsilon\colon \widetilde{Y}\to Y$ be a blowup at $C_0\cap Z$ with $\mathcal{E}$ be the associated (-1) curve. Let $\overline{C_0}, \overline{Z}$ be the strict transformations of $C_0, Z$. See Figure \ref{fig:ex2} for the configuration of curves on $Y$.

		Denote by $L$ the divisor $K_{\widetilde{Y}}+\overline{Z}+\frac{1}{3}\overline{C_0}$.
		Note that $L\equiv n\epsilon^*F+\mathcal{E}+\overline{Z}+\frac{1}{3}\overline{C_0}\ge 0$. Moreover, 
		$L\mathcal{E}=\frac{1}{3}$, $L\overline{Z}=L\overline{C_0}=0$ and $L^2=L(n\epsilon^*F)+L\mathcal{E}=n+\frac{1}{3}>0$. Therefore, $L$	
		is nef and big.  
		Let $R:=\mathop{\oplus}\limits_{i=0}^{\infty}H^{0}(\tilde{Y},\lfloor i( K_{\widetilde{Y}}+\overline{Z}+\frac{1}{3}\overline{C_0})\rfloor)$, which is a finitely generated graded ring by \cite[Thm 1.2]{BCHM10}. Define $X\colon =\mathrm{Proj}~ R$, which is a normal projective surface. There exists a natural birational morphism 
		$\pi\colon \tilde{Y} \to X$. 
		Note that $\pi_*(\overline{Z}+\frac{1}{3}\overline{C_0})=0$.	
		In fact, $(X,0)$ is the (log) canonical model of the pair $(\widetilde{Y}, \overline{Z}+\frac{1}{3}\overline{C_0})$. We refer to \cite{AD17} \cite{SMMP} for log canonical models.

		Assume $D$ is a reduced irreducible curve  with $LD=0$ and $D\not= \overline{Z},\overline{C_0}$. Then $D\epsilon^*F=D\mathcal{E}=D\overline{Z}=D \overline{C_0}=0$, which means $D$ is a (-2) curve contained in a reducible fiber disjoint with $\overline{Z}$ and $\overline{C_0}$.
		Hence $\pi$ contracts $\overline{Z},\overline{C_0}$ and all (-2) curves $D$ contained in reducible fibers disjoint with $\overline{Z}$ and $\overline{C_0}$ into a simple elliptic singularity, a $\frac{1}{3}(1,1)$ singularity and $ADE$ singularities on $X$. 
		
		Therefore, $X$ is a KSBA stable surface $X$ with $K_X^2=p_g+\frac{1}{3}$, which is an example of Theorem \ref{main} (B). 
		\begin{center}
			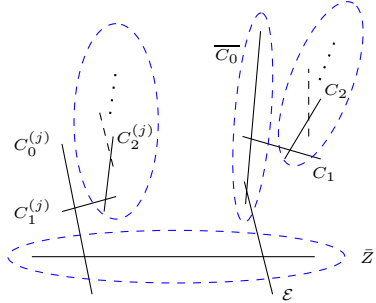
\begin{figure}[H]\label{fig:ex2}
			\begin{tikzpicture}[font=\tiny]
				
				\begin{scope}[xshift=0cm,xscale=0.8]
					
					\draw (2,0)--(6.7,0); 
					\draw (6,-0.5)--(5.53,1); 
					\draw (3,-0.5)--(2.5,1.5); 
					\draw (5.55,0.7)--(5.8, 3); 
					\draw (5.5,1.6)--(6.8,1.3);
					\draw (6.2,1.3)--(6.8,2.1); 
					\draw[dashed] (6.6,1.6)--(6.6,2.5); 
					\draw[loosely dotted,thick] (6.78,2.35)--(7.05,2.9); 
					\draw (2.5, 0.6)--(3.4, 0.8);  
					\draw (3.2, 0.6)--(3.35, 1.6);  
					\draw[dashed] (3.35, 1.2)--(3.1, 2.0);  
					\draw[loosely dotted, thick] (3.3, 1.9)--(3.4, 2.5);  
					
					\coordinate (O1) at (6.8,2.3);  
					\coordinate (O2) at (3.4, 1.8); 
					\coordinate (O3) at (4.4, 0); 	
					\coordinate (O4) at (5.67, 1.85);
					
					\draw[dashed,blue, rotate around={63:(O1)}] (O1) ellipse (1.2 and 0.5);
					\draw[dashed,blue, rotate around={90:(O2)}] (O2) ellipse (1.3 and 0.7);
					\draw[dashed,blue] (O3) ellipse (2.8 and 0.35);
					\draw[dashed,blue, rotate around={84:(O4)}] (O4) ellipse (1.4 and 0.3);
					\node[right] at (7.3,0) {$\bar{Z}$};
					\node[right] at (6,-0.5) {$\mathcal{E}$};
					\node[left] at (5.6,2.7) {$\overline{C_0}$};
					\node[right] at (6.5,1.1){$C_1$};
					\node[right] at (6.7,2.2){$C_2$};
					\node[left] at (2.5,1.5) {$C^{(j)}_0$};
					\node[left] at (2.5, 0.6) {$C^{(j)}_1$};
					\node[right] at (3.25, 1.6){$C^{(j)}_2$};
				\end{scope}
			\end{tikzpicture}
		\caption{configuration of curves on $Y$ for Example \ref{6.1b}}
		\end{figure}
		\end{center}
		
	\end{example}

	\begin{example}\label{6.1c0}
		Let $Y$ be a smooth minimal Jacobian surface over $\mathbb{P}^1$ with $E_0$ a zero section and $K_Y\equiv nF$. 
		Let $X\colon = \mathrm{Proj} \mathop \oplus\limits_{i=0}^{\infty} H^0(Y, \lfloor i( K_Y+ \frac{n}{n+2}E_0)\rfloor)$. 
		Then $X$ is a KSBA stable surface $X$ with $K_X^2=\frac{p_g-1}{p_g+1}(p_g-1)$, which is an example of Theorem \ref{main} (C). In fact, $X$ is obtained from $\widetilde{Y}$ by contracting $E_0$ and all (-2) curves $D$ contained in reducible fibers disjoint with $E_0$. 
		
		We remark that this surface has occurred in  \cite[Example 1.8]{TZ92}. 
		
	\end{example}
	
	\begin{example}\label{6.1c1}
		Let $Y$ be a smooth minimal Jacobian surface over $\mathbb{P}^1$ with $E_0$ a zero section and $K_Y\equiv nF$. Assume further $Y$ has a reducible fiber $F_0=\sum_{i=0}^{m} C_i$. We may assume $C_0Z=1$ and $C_iZ=0$, $i=1,...,m$. 
		Let $X\colon = \mathrm{Proj} \mathop \oplus\limits_{i=0}^{\infty} H^0(Y, \lfloor i( K_Y+ \frac{2n}{2n+3}E_0+\frac{n}{2n+3}C_0)\rfloor)$. 
		Then $X$ is a KSBA stable surface $X$ with $K_X^2=\frac{2p_g-2}{2p_g+1}(p_g-1)$, which is an example of Theorem \ref{main} (C). In fact, $X$ is obtained from $\widetilde{Y}$ by contracting $E_0, C_0$ and all (-2) curves $D$ contained in reducible fibers disjoint with $E_0, C_0$. 
		
			\begin{center}
			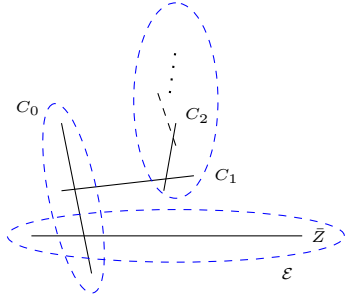
\begin{figure}[H]\label{fig:ex2}
				\begin{tikzpicture}[font=\tiny]
					
					\begin{scope}[xshift=0cm,xscale=0.8]
						
						\draw (2,0)--(6.5,0); 
						\draw (3,-0.5)--(2.5,1.5); 
						\draw (2.5, 0.6)--(3.4+1.3, 0.8);  
						\draw (3.2+1, 0.6)--(3.4+1, 1.5);  
						\draw[dashed] (3.4+1, 1.2)--(3.1+1, 1.9);  
						\draw[loosely dotted, thick] (3.3+1, 1.9)--(3.4+1, 2.5);  
						\coordinate (O0) at (2.7,0.5);  
						\coordinate (O2) at (3.4+1, 1.8); 
						\coordinate (O3) at (4.4, 0); 	
						
						\draw[dashed,blue, rotate around={105:(O0)}] (O0) ellipse (1.3 and 0.4);
						\draw[dashed,blue, rotate around={90:(O2)}] (O2) ellipse (1.3 and 0.7);
						\draw[dashed,blue] (O3) ellipse (2.8 and 0.35);
						\node[right] at (6.5,0) {$\bar{Z}$};
						\node[right] at (6,-0.5) {$\mathcal{E}$};
						\node[left] at (2.3,1.7) {$C_0$};
						\node[left] at (3.4+2.2, 0.8) {$C_1$};
						\node[right] at (3.4+1, 1.6){$C_2$};
					\end{scope}
				\end{tikzpicture}
				\caption{configuration of curves on $Y$ for Example \ref{6.1c1}}
			\end{figure}
		\end{center}
	
	\end{example}
	
	\begin{example}\label{6.1c2}
		Let $Y$ be a smooth minimal Jacobian surface over $\mathbb{P}^1$ with $E_0$ a zero section and $K_Y\equiv nF$ ($n\ge 4$). Assume further $Y$ has a reducible fiber $F_0=\sum_{i=0}^{m} C_i$. We may assume $C_0Z=1$ and $C_iZ=0$, $i=1,...,m$. 
		Let $\epsilon\colon \widetilde{Y}\to Y$ be a blowup at $C_0\cap E_0$ and $\mathcal{E}$ be the (-1) curve. Let $\overline{C_0}, \overline{E_0}$ be the strict transformations of $C_0, E_0$. 
		Let $X\colon = \mathrm{Proj} \mathop \oplus\limits_{i=0}^{\infty} H^0(\tilde{Y}, \lfloor i( K_{\tilde{Y}}+\frac{n+1}{n+3}\overline{E_0}+\frac{1}{3}\overline{C_0})\rfloor)$.
		Then $X$ is a KSBA stable surface $X$ with $K_X^2=\frac{(3p_g-2)p_g-4}{3(p_g+2)}$, which is an example of Theorem \ref{main} (C). In fact, $X$ is obtained from $\widetilde{Y}$ by contracting $\overline{C_0}, \overline{E_0}$ and all (-2) curves $D$ contained in reducible fibers disjoint with $\overline{E_0},\overline{C_0}$. 
		
		\begin{center}
			\begin{figure}[H]\label{fig:ex2}
				\begin{tikzpicture}[font=\tiny]
					
					\begin{scope}[xshift=0cm,xscale=0.8]
						
						\draw (2,0)--(6.7,0); 
						\draw (6,-0.5)--(5.53,1); 
						\draw (5.55,0.7)--(5.8, 3); 
						\draw (5.5,1.6)--(6.8,1.3);
						\draw (6.2,1.3)--(6.8,2.1); 
						\draw[dashed] (6.6,1.6)--(6.6,2.5); 
						\draw[loosely dotted,thick] (6.78,2.35)--(7.05,2.9); 
						
						\coordinate (O1) at (6.8,2.3);  
						\coordinate (O3) at (4.4, 0); 	
						\coordinate (O4) at (5.67, 1.85); 
						
						\draw[dashed,blue, rotate around={63:(O1)}] (O1) ellipse (1.2 and 0.5);
						\draw[dashed,blue] (O3) ellipse (2.8 and 0.35);
						\draw[dashed,blue, rotate around={84:(O4)}] (O4) ellipse (1.4 and 0.3);
						\node[right] at (7.3,0) {$\overline{E_0}$};
						\node[right] at (6,-0.5) {$\mathcal{E}$};
						\node[left] at (5.6,2.7) {$\overline{C_0}$};
						\node[right] at (6.5,1.1){$C_1$};
						\node[right] at (6.7,2.2){$C_2$};
					\end{scope}
				\end{tikzpicture}
				\caption{configuration of curves on $\tilde{Y}$ for Example \ref{6.1c2}}
			\end{figure}
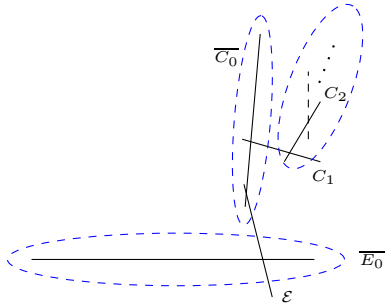
		\end{center}
	\end{example}
	
	{\bf Acknowledgements: }{\rm We thank the referees for their time and
		comments. }

	\bibliographystyle{alpha}

\begin{thebibliography}{plain}
		
		\bibitem[Ale89]{alexeev89}
		Valery Alexeev.
		\newblock Fractional indices of log Del Pezzo surfaces.
		\newblock Math. USSR Izv. 8 (1989), 613-629.
		
		\bibitem[Ale96]{alexeev96a}
		Valery Alexeev.
		\newblock Moduli spaces {$M\sb{g,n}(W)$} for surfaces.
		\newblock In {\em Higher-dimensional complex varieties (Trento, 1994)}, pages
		1--22. de Gruyter, August 1996.
		
		\bibitem[Ale94]{alexeev94}
		Valery Alexeev.
		\newblock Boundedness and $K^2$ for log surfaces
		\newblock International Journal of Mathematics 1994 05:06, 779-810.
		
		\bibitem[Ale06]{alexeev06}
		Valery Alexeev.
		\newblock Higher-dimensional analogues of stable curves.
		\newblock In {\em International {C}ongress of {M}athematicians. {V}ol. {II}},
		pages 515--536. Eur. Math. Soc., Z{\"u}rich, 2006.
		
		\bibitem[AP13]{AP13}
		Alexeev, V.,  Pardini, R. 
		\newblock \emph{On the existence of ramified abelian covers} (arXiv:1210.6174). arXiv. https://doi.org/10.48550/arXiv.1210.6174
		
		
		\bibitem[AL19]{AL19}
		V. Alexeev and W. Liu. 
		\newblock Open surfaces of small volume. 
		\newblock Algebraic Geometry 6 (2019), 312–327.
		
		
		\bibitem[Ant19]{ant19}
		Anthes, B. 
		\newblock \emph{Gorenstein stable surfaces with $K_X^2 = 2$ and $\chi(\mathcal O_X) = 4$} (arXiv:1904.12307). arXiv. https://doi.org/10.48550/arXiv.1904.12307
		
		\bibitem[AD17]{AD17}
		Ascher, Kenneth, and Dori Bejleri. 
		\newblock \emph{ Log Canonical Models of Elliptic Surfaces.}
		\newblock Advances in Mathematics, vol. 320, Nov. 2017, pp. 210–43. DOI.org (Crossref), https://doi.org/10.1016/j.aim.2017.08.035.
		
		
		\bibitem[BCHM10]{BCHM10}
		Caucher Birkar, Paolo Cascini, Christopher D. Hacon, and James McKernan.
		\newblock \emph{Existence of minimal models for varieties of log general type.}
		\newblock Journal of the American Mathematical Society, 23(2), 405–468. https://doi.org/10.1090/S0894-0347-09-00649-3
		
		\bibitem[BHPV]{BHPV}
		Barth, Wolf, et al.
		\newblock \emph{Compact complex surfaces.} Vol. 4. Springer, 2015.
		
		\bibitem[Bla95]{Bla95}
		Blache R.
		\newblock Riemann-Roch theorem for normal surfaces and applications[J].
		\newblock Abhandlungen Aus Dem Mathematischen Seminar Der Universit\"at Hamburg, 1995, 65(1):307-340.
		
		\bibitem[Bom73]{Bom73}
		Bombieri, E.
		\newblock Canonical models of surfaces of general type.
		\newblock Publications Math{\'e}matiques de l'Institut des Hautes {\'E}tudes Scientifiques (1973) 42: 171.
		
		\bibitem[FM94]{F-M94}
		Friedman R, Morgan J W.
		\newblock \emph{Smooth four-manifolds and complex surfaces.}
		\newblock Ergebnisse Der Mathematik Und Ihrer Grenzgebiete, 1994, 27.
		
		
		
		\bibitem[FPR15a]{FPR15a}
		Franciosi, Marco, R. Pardini, and S. Rollenske.
		\newblock  Gorenstein stable surfaces with $K_X^2 = 1$ and $p_g>0$.
		\newblock  Mathematics (2015).
		
		\bibitem[FPR15b]{FPR15b}
		M.Franciosi,  R. Pardini, S. Rollenske.
		\newblock Log-canonical pairs and Gorenstein stable surfaces with $K_X^2=1$. \newblock Mathematics 151.8(2015):1529-1542.
		
		\bibitem[Gab02]{GAB02}
		La Nave, Gabriele.
		\newblock  Explicit stable models of elliptic surfaces with sections.
		\newblock  2002, arXiv:math/0205035.
		
		
		\bibitem[Hor76a]{Hor76a}
		Horikawa E.
		\newblock  Algebraic surfaces of general type with small $c^2_1$. I,
		\newblock  Ann. of Math. 104(1976), 357-387. 35(1-2):391-406.
		
		\bibitem[Hor76a]{Hor76b}
		Horikawa E.
		\newblock Algebraic surfaces of general type with small $c^2_1$. II,
		\newblock  Invent. Math. 37(1976):121-155.
		
		\bibitem[Hor78]{Hor78}
		Horikawa, Eiji.
		\newblock Algebraic surfaces of general type with small $c^2_1$. IV.
		\newblock Inventiones mathematicae  50.2(1978):103-128.
		
		\bibitem[Kol12]{kollar12}
		Jan\'os Koll\'ar.
		\newblock Moduli of varieties of general type.
		\newblock In G.~Farkas and I.~Morrison, editors, {\em Handbook of Moduli:
			Volume II}, volume~24 of {\em Advanced Lectures in Mathematics}, pages
		131--158. International Press, 2012.
		
		\bibitem[Kol13]{SMMP}
		J{\'a}nos Koll{\'a}r.
		\newblock {\em Singularities of the minimal model program}, volume 200 of {\em
			Cambridge Tracts in Mathematics}.
		\newblock Cambridge University Press, Cambridge, 2013.
		\newblock With a collaboration of S{\'a}ndor Kov{\'a}cs.
		
		\bibitem[Kol14]{kollarModuli}
		J{\'a}nos Koll{\'a}r.
		\newblock {\em Moduli of varieties of general type}.  
		\newblock 2014.
		\newblock book in preparation.
		
		\bibitem[KM98]{Kollar-Mori}
		J{\'a}nos Koll{\'a}r and Shigefumi {M}ori.
		\newblock {\em Birational geometry of algebraic varieties}, volume 134 of {\em
			Cambridge Tracts in Mathematics}.
		\newblock Cambridge University Press, Cambridge, 1998.
		\newblock With the collaboration of C. H. Clemens and A. Corti, Translated from
		the 1998 Japanese original.
		
		\bibitem[KSB88]{KSB88}
		J{\'a}nos Koll{\'a}r, and Nicholas I. Shepherd-Barron.
		\newblock \emph{Threefolds and deformations of surface singularities.}
		\newblock Inventiones mathematicae 91.2 (1988): 299-338.
		
		\bibitem[KSC04]{KSC04}
		J{\'a}nos Koll{\'a}r, K. E. Smith, and A. Corti.
		\newblock Rational and nearly rational varieties.
		\newblock Cambridge University Press Cambridge (2004):vi,235.
		
		\bibitem[Lau77]{laufer77}
		Henry~B. Laufer.
		\newblock On minimally elliptic singularities.
		\newblock {\em Amer. J. Math.}, 99(6):1257--1295, 1977.
		
		\bibitem[Liu22]{Liu22}
		Liu, W.
		\newblock \emph{The minimal volume of log surfaces of general type with positive geometric genus} (arXiv:1706.03716). arXiv. https://doi.org/10.48550/arXiv.1706.03716
		
		\bibitem[LR12]{LR12}
		Wenfei Liu, S. Rollenske. (2012).
		\newblock   Two-dimensional semi-log-canonical hypersurfaces.
		\newblock  Le Matematiche 67.2(2012):185-202.
		
		\bibitem[LR13]{LS13}
		Liu, Wenfei, and S\"{o}nke Rollenske.
		\newblock \emph{Geography of Gorenstein stable log surfaces.}
		\newblock  Transactions of the American Mathematical Society 368.4(2016).
		
		\bibitem[Nag60]{Nag60}
		Nagata, M.
		\newblock \emph{On rational surfaces I}.
		\newblock Memoirs of the College of Science University of Kyoto, 33,(1960):351-370.
		
		\bibitem[Pin74]{pink74}
		Pinkham, H.
		\newblock \emph{ Deformations of algebraic varieties with $G_m$ action}. 
		\newblock Paris: Soci{\'e}t{\'e} math{\'e}matique de France. Ast{\'e}risque, no. 20 (1974), 140 p.   
		
		
		\bibitem[Rei97]{Reid97}
		Reid, Miles.
		\newblock \emph{Chapters on algebraic surfaces.}
		\newblock Complex algebraic geometry, IAS/Park City Mathematics Series. 1997.
		
		\bibitem[RR22]{RR22}
		Rana, J.,  Rollenske, S. 
		\newblock \emph{Standard stable Horikawa surfaces} (arXiv:2211.12059). (2022).  arXiv. https://doi.org/10.48550/arXiv.2211.12059
		
		\bibitem[TZ92]{TZ92}
		Tsunoda Shuichiro, and D. Q. Zhang.
		\newblock Noether's inequality for non-complete algebraic surfaces of general type.
		\newblock Publications of the Research Institute for Mathematical Sciences 28.1(1992):21-38.
		
		\bibitem[Fuj75]{Fuj75}
		T. Fujita.
		\newblock {\em On the structure of polarized varieties with $\Delta$-genera zero}.
		\newblock J. Fac. Sci. Univ. of Tokyo 22(1975):103-115.
		
		
		%
		
		
	\end{thebibliography}

\end{document}